\documentclass{amsart}
\usepackage{amsthm}
\usepackage{amsfonts}
\usepackage{mathrsfs}
\usepackage{graphicx}
\usepackage{amsfonts}
\usepackage[dvipsnames,usenames]{color}
\usepackage{amsmath, amsthm, amsfonts, amssymb}
\usepackage{subfigure}

\usepackage{young}
\usepackage{comment}

\theoremstyle{definition}
\numberwithin{equation}{section}
\theoremstyle{remark} \hsize=7.5truein \vsize=8.6truein

\def\E{\mathbb{E}}

\def\be{{\bf{e}}}
 at 10 true pt

\def\be#1{ \begin{equation}\label{#1} }
\def\bas{\begin{align*}}
\def\eas{\end{align*}}
\def\bi{\begin{itemize}}
\def\ei{\end{itemize}}

\def\emph#1{{\it #1}}
\def\textbf#1{{\bf #1}}

\def \small {\scriptsize}

\def\tr{\mathrm{Tr}}

\theoremstyle{plain}
  \newtheorem{theorem}[subsection]{Theorem}

  \newtheorem{proposition}[subsubsection]{Proposition}
  
  \newtheorem{lemma}[subsection]{Lemma}

\theoremstyle{remark}
  \newtheorem{remark}[subsection]{\bf Remark}

\theoremstyle{definition}
  
\begin{document}
\title[Singular Sample Covariance Matrices]{New Methods for Handling Singular Sample Covariance Matrices}
\author{Gabriel H. Tucci}
\address{Gabriel H. Tucci is the global head of Central Risk and Cash Trading in the Equities division at Citi, 388 Greenwich Street, New York, NY 10013, USA. }
\email{gabrieltucci@gmail.com}
\thanks{This paper was presented in part at  the International Symposium on Information Theory, Boston, 2012.}

\author{Ke Wang}
\address{Ke Wang is with Department of Mathematics, Hong Kong University of Science and Technology, Clear Water Bay, Kowloon, Hong Kong}
\email{kewang@ust.hk}
\thanks{Ke Wang is supported by by HKUST Initiation Grant IGN16SC05.}

\begin{abstract}
The estimation of a covariance matrix from an insufficient amount of data is one of the most common problems in fields as diverse as multivariate statistics, wireless communications, signal processing, biology, learning theory and finance. In a joint work of Marzetta, Tucci and Simon, a new approach to handle singular covariance matrices was suggested. The main idea was to use dimensionality reduction in conjunction with an average over the Stiefel manifold. In this paper we continue with this research and we consider some new approaches to handle this problem. One of the methods is called the mean conjugate estimator under Ewens measure and uses a randomization of the sample covariance matrix over all the permutation matrices with respect to the Ewens measure. The techniques used to attack this problem are broad and run from random matrix theory to combinatorics.

\end{abstract}

\maketitle

\section{Introduction}

The estimation of a covariance matrix from an insufficient amount of data is one of the most common problems in fields as diverse as multivariate statistics, wireless communications, signal processing, biology, learning theory and finance. For instance, the covariation between asset returns plays a crucial role in modern finance. The covariance matrix and its inverse are the key statistics in portfolio optimization and risk management. Many recent financial innovations involve complex derivatives, like exotic options written on the minimum, maximum or difference of two assets, or some structured financial
products, such as CDOs. All of these innovations are built upon, or in order to exploit, the correlation structure of two or more assets. In the field of wireless communications, covariance estimates allows us to compute the direction of arrival (DOA), which is a critical task in smart antenna systems since it enables accurate mobile location (see \cite{SN1,SN2}). Another application is in the field of biology and involves the interactions between proteins or genes in an organism and the joint time evolution of their interactions (see \cite{SS} for instance). 

Typically the covariance matrix of a multivariate random variable is not known but has to be estimated from the data. Estimation of covariance matrices then deals with the question of how to approximate the actual covariance matrix on the basis of samples from the multivariate distribution. Simple cases, where the number of observations is much greater than the number of variables, can be dealt with by using the sample covariance matrix. In this case, the sample covariance matrix is an unbiased and efficient estimator of the {\it true} covariance matrix. However, in many practical situations we would like to estimate
the covariance matrix of a set of variables from an insufficient amount of data. In this case the sample covariance matrix is singular (non--invertible) and therefore a fundamentally bad estimate. More specifically, let $X$ be a random vector $X=(X_1, \ldots, X_m)^T \in \mathbb{C}^{m\times 1}$ and assume for simplicity that $X$ is centered. Then the true covariance matrix is given by 
\begin{equation}
\Sigma = \mathbb{E} (X X^*)=(\mathrm{cov}(X_i, X_j))_{1\le i,j \le m}.
\end{equation}
Consider $n$ independent samples or realizations $x_1, \ldots, x_n \in \mathbb{C}^{m}$ and form the $m \times n$ data matrix $M=(x_1, \ldots, x_n)$. Then the sample covariance matrix is an $m\times m$ non--negative definite matrix defined as 
\begin{equation}
K=\frac{1}{n} M M^*.
\end{equation}
If $n\rightarrow +\infty$ and $m$ is fixed, then the sample covariance matrix $K$ converges (entrywise) to $\Sigma$ almost surely. Whereas, as we mentioned before, in many empirical problems, the number of measurements is less than the dimension $(n<m)$, and thus the sample covariance matrix is singular. Our objective in this paper is to recover the true covariance matrix $\Sigma$ from $K$ under the condition $n<m$.

The conventional treatment of covariance singularity artificially converts the singular sample covariance matrix into an invertible (positive definite) covariance by the simple expedient
of adding a positive diagonal matrix, or more generally, by taking a linear combination of the sample covariance and the identity matrix. This procedure is variously called ``diagonal loading'' or ``ridge regression'' \cite{RNE, DS98}. This one is defined as $\alpha K+ \beta I_m$ where $\alpha$ and $\beta$ are called loading parameters. The resulting matrix is positive definite, invertible and preserves the eigenvectors of the sample covariance. The eigenvalues of $\alpha K+ \beta I_m$ are a uniform rescaling and shift of the eigenvalues of $K$. There are many methods in choosing the optimum loading parameters, see \cite{LW02}, \cite{mestre08} and \cite{MD}. On the other hand, if the true covariance matrix is assumed to have some level of sparsity, several works have been established, such as the banding and thresholding methods studied by Bickel and Levina \cite{BL08a, BL08b}, Wu and Pourahmadi \cite{WP}, El Karoui \cite{EK} and Rothman et al. \cite{RBLZ}, to mention a few. In more recent works, Cai, Zhang and Zhou \cite{CZZ} and Cai and Zhou \cite{CZ} derive the optimal rate of convergence for estimating the true covariance matrix and its inverse under operator norm, Frobenius norm and $l_1$ norm, for a large range of sparse covariance matrices.

In Marzetta, Tucci and Simon's paper \cite{MTS} a new approach to handle singular covariance matrices was suggested. They use the idea of \emph{random dimension reduction}. Let $p\leq n$ be a parameter, to be estimated later, and consider the set of all $p\times m$ one-sided unitary matrices 
\begin{equation}\label{eq:set}
\Omega_{p,m}=\{\Phi \in \mathbb{C}^{p\times m}~:~ \Phi \Phi^* = I_p \}.
\end{equation} 
This set has a manifold structure and is called the Stiefel manifold. Note that $\Phi M$, that is the multiplication of the one-sided unitary matrix $\Phi$ with the data matrix $M$, results in a new data matrix with reduced dimension. And 
\begin{align}\label{eq:sam}
\frac{1}{n} (\Phi M)(\Phi M)^*=\Phi K \Phi^*
\end{align}
can be viewed as a new sample covariance matrix of size $p$. Then $\Phi^* (\Phi K \Phi^*) \Phi$ will project the data back to $n$-dimensional space. In \cite{MTS}, they endow the Stiefel manifold with the \emph{Haar measure}, that is, the uniform distribution on the set $\Omega_{p,m}$. Further, they define the operators 
\begin{equation*}
\mathrm{cov}_p(K) = \mathbb{E} (\Phi^* (\Phi K \Phi^*) \Phi);
\end{equation*}
\begin{equation*}
\mathrm{invcov}_p(K) = \mathbb{E} (\Phi^* (\Phi K \Phi^*)^{-1} \Phi),
\end{equation*}
where the expectation is taken with respect to the Haar measure. The operators $\mathrm{cov}_p(K)$ and $\mathrm{invcov}_p(K)$ are used to estimate the true covariance matrix $\Sigma$ and its inverse $\Sigma^{-1}$ respectively. It was found that 
$$
\mathrm{cov}_p(K) = \frac{p}{(m^2-1)m} \Big( (mp-1) K + (m-p) \mathrm{Tr}(K) I_m \Big),
$$ 
which is the same as diagonal loading. Moreover, they investigated the properties of $\mathrm{invcov}_p(K)$. If $K$ is decomposed as $K= U D U^*$, with $D=\mathrm{diag}(d_1,\ldots,d_n,0,\ldots,0)$, then 
$$
\mathrm{invcov}_p(K)= U \mathrm{invcov}_p(D) U^*,
$$
and
\begin{align}\label{eq:invD}
\mathrm{invcov}_p(D)=\mathrm{diag} (\lambda_1, \ldots, \lambda_n, \mu, \ldots, \mu).
\end{align}
In other words, $\mathrm{invcov}_p(K)$ preserves the eigenvectors of $K$, and transforms all the zero eigenvalues to a non--zero constant value. They also provided formulas to compute the values of $\lambda_i$ and $\mu$, and studied their asymptotic behavior using techniques from free probability. 

The explicit formula of $\lambda_i$'s of $\mathrm{invcov}_p(D)$ in \eqref{eq:invD} is derived in \cite{MTS} as a partial derivative of a rather complicated integral (see (11) and Theorem 1 in \cite{MTS}). In this paper, we further investigate the properties of the $\mathrm{invcov}_p(K)$ or equivalently the $\mathrm{invcov}_p(D)$ operators. These results are presented in Section \ref{sec:invcov}. We first show that $\mathrm{invcov}_p(D)$ has a surprisingly simple algebraic structure, i.e. it is a polynomial of the diagonal matrix $D$. We also provide formulas to compute the coefficients of the polynomial and illustrate the computation through a small dimensional example in Appendix \ref{sec:small}. The formulas involve complicated combinatorial subjects and thus make further investigation on the performance, i.e. optimize the error functions with respect to the parameters, rather difficult. 

Therefore, it is natural to look for alternative random operators that are easy to compute, analyze and implement. It is known that a random unitary matrix with Haar measure behaves asymptotically like a random uniform permutation matrix (see \cite{Wieand1} and \cite{Wieand2}). Our first attempt is to conjugate the sample covariance matrix $K$ with a permutation matrix $M_{\sigma}$. In \cite{Viana}, the mean conjugate $K_1=\mathbb{E}(M_{\sigma} K M_{\sigma}^T)$ of a square matrix $K$ averaging over uniform permutation matrix $M_{\sigma}$ is studied. It is found in \cite{Viana} that $K_1$ is always a scalar multiple of identity matrix plus a rank-one matrix (see Remark \ref{rem:uniform}), which is a well-conditioned matrix in most cases. 

Now we investigate the mean conjugate of a matrix $K$ under a generalized measure on the permutation group, called the Ewens measure with parameter $\theta>0$ (see \eqref{eq:ewens} below). We obtain a closed form expression for the estimator $K_\theta=\mathbb{E}(M_{\sigma} K M_{\sigma}^T)$ in Theorem \ref{permutation} using combinatorial techniques. We find that the averaging operation on diagonal matrices is equivalent to the conventional diagonal loading (see Remark \ref{rem:diag}). For the matrix $K$ with certain structures, the averaging over all permutation matrices under Ewens measure by choosing $\theta$ propositional to the dimension $m$, is asymptotically equivalent to \emph{linear shrinkage estimator} proposed by Lenoit and Wolf \cite{LW04}. This result is proved in Section \ref{sec:asymptotic}. We propose this new method to estimate the covariance matrices and call it the \emph{mean conjugate estimator under Ewens measure}.


In Section \ref{sec:mix}, we extend the ideas of constructing the $\mathrm{cov}_p(K)$ and $\mathrm{invcov}_p(K)$ operators by replacing random unitary matrices with random permutation matrices.  We first extend the definition of permutation matrices to get $p\times m$ unitary matrices $V_{\sigma}$ and use the Ewens measure  in Section \ref{sec:ewens}. Then we define two new operators 
$$
K_{\theta,m,p}:= \mathbb{E} \big( V_{\sigma}^T (V_{\sigma} K V_{\sigma}^T) V_{\sigma} \big)
$$
$$
\tilde{K}_{\theta,m,p}:=\mathbb{E}\big(V_\sigma^T(V_\sigma K V_\sigma^T)^{+}V_\sigma\big)
$$
to estimate $\Sigma$ and $\Sigma^{-1}$ respectively. Here $A^+$ is the \emph{Moore-Penrose pseudo inverse} of the $A$. If $A$ is an $m\times n$ complex or real matrix, then $A^+$ is an $n\times m$ complex or real matrix that satisfies $AA^+$ and $A^+ A$ are both Hermitian or symmetric, $AA^+A=A$ and $A^+ A A^+=A$. For any matrix $A$, the pseudo inverse $A^+$ always exists. We provide an explicit formula for $K_{\theta,m,p}$ and an inductive formula to compute $\tilde{K}_{\theta,m,p}$.

In Section \ref{sec:simulation}, we first study the asymptotic behavior for certain matrices with the mean conjugate estimator under Ewens measure.  We conduct some simulation study focusing on the mean conjugate estimator under Ewens measure. However, we do not include the simulations on the hybrid operators $K_{\theta,m,p}$ and $\tilde{K}_{\theta,m,p}$ since currently we do not have adequate understanding on them from explicit formulas obtained in Section \ref{sec:mix}.

{\bf Notation:} Throughout this paper, $\mathbf{1}_S$ is the indicator function of a set $S$. We sometimes use $[n]$ to present the set $\{1,2,\ldots,n\}$, and $\mathrm{Tr}(A)$ is the trace of a matrix $A$. For an $m \times m$ matrix $A$, we use the (normalized) Frobenius norm $\| A \|_F = \frac{1}{\sqrt{m}}\sqrt{ \mathrm{Tr}(A A^*) }$. We denote $A^{+}$ the Moore-Penrose pseudo inverse of the matrix $A$.  For a vector $v=(v_1,\ldots,v_m)$, we use the Euclidean norm $\| v \|_2 = \sqrt{\sum_{i=1}^m |v_i |^2}$.  We use $v(k)$ to denote the $k$th entry of $v$. We use $\mathbf{e}=(1,\ldots,1)^T$ to represent the all-one vector and $e_i$ are the standard basis vectors.  We use the notation $\kappa \vdash n$ to indicate that $\kappa$ is an integer partition of the positive integer $n$.

{\bf Acknowledgement:} We would like to thank the anonymous referees for their careful reading and many insightful suggestions.

\section{Some Properties of the $\mathrm{invcov}_p$ Estimator}\label{sec:invcov}
We first collect some preliminaries about Schur polynomials that will be needed later in studying the properties of the $\mathrm{invcov}_p$ estimator.
\subsection{Preliminaries of Schur polynomials}\label{sec_prelim}
A symmetric polynomial is a polynomial $P(x_1,x_2,\ldots, x_n)$ in $n$ variables such that if any of the variables are interchanged one obtains the same polynomial. Formally, $P$ is a symmetric polynomial if for any permutation $\sigma$ of the set $\{1,2,\ldots, n\}$ one has that 
$$
P(x_{\sigma(1)},x_{\sigma(2)},\ldots, x_{\sigma(n)})=P(x_1,x_2,\ldots, x_{n}).
$$
Symmetric polynomials arise naturally in the study of the relation between the roots of a polynomial in one variable and its coefficients, since the coefficients can be given by a symmetric polynomial expressions in the roots. Symmetric polynomials also form an interesting structure by themselves. The resulting structures, and in particular the ring of symmetric functions, are of great importance in combinatorics and in representation theory (see for instance \cite{Fulton, Muir, Mac, Sagan} for more on details on this topic).

The Schur polynomials are certain symmetric polynomials in $n$ variables. This class of polynomials is also very important in representation theory since they are the characters of irreducible representations of the general linear groups. The Schur polynomials are indexed by partitions. A partition of a positive integer $n$, also called an integer partition, is a way of writing $n$ as a sum of positive integers. Two partitions that differ only in the order of their summands are considered to be the same partition. Therefore, $\kappa=(\kappa_1,\ldots,\kappa_n)\vdash n$ is a partition of a positive integer of $n$ if 
$$
\sum_{i=1}^{n}{\kappa_i}=n \hspace{0.5cm} \text{with} \hspace{0.5cm} \kappa_{1}\geq \kappa_2\geq\ldots\geq \kappa_{n}\geq 0.
$$ 
The $\kappa_i$'s are called the \emph{parts} of $\kappa$. Notice that some of the $\kappa_i$ could be zero. Sometimes, we use another equivalent way to represent a partition. We write $\kappa=(1^{r_1},2^{r_2},\ldots,n^{r_n})\vdash n$ where $r_i$ is the number of $i$ appearing as parts in $\kappa$. Thus $\sum_{i=1}^n i\cdot r_i=n$. Integer partitions are usually represented by the so called Young's diagrams (also known as Ferrers' diagrams). A Young diagram is a finite collection of boxes, or cells, arranged in left--justified rows, with the row lengths weakly decreasing (each row has the same or shorter length than its predecessor). Listing the number of boxes on each row gives a partition $\kappa$ of a non-negative integer $n$, the total number of boxes of the diagram. The Young diagram is said to be of shape $\kappa$, and it carries the same information as that of partition. For instance, in Figure \ref{young} we can see the Young diagram corresponding to the partition $(5,4,1)$ of the number 10.
\begin{figure}[!ht]
  \begin{center}
    \includegraphics[width=3cm]{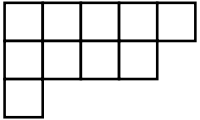}
    \caption{\small{Young digram representation of the partition $(5,4,1)$.} }
    \label{young}
  \end{center}
\end{figure}
Given a partition $\kappa$ of $m$. Assume $m\ge n$. The Schur polynomial of shape $\kappa$ in the variables $(d_1,\ldots,d_n)$ is defined as 
$$ 
s_\kappa (d_1,\ldots,d_n) = \frac{\mathrm{det} (d_i^{n+\kappa_j -j })_{i,j=1}^n }{\mathrm{det} (d_i^{n-j})_{i,j=1}^n}.
$$ 
Indeed the denominator $\mathrm{det} (d_i^{n-j})_{i,j=1}^n$  is the determinant of the Vandermonde matrix
\begin{align}\label{eq:vandermonde}
\Delta(d_1,\ldots,d_n)=\begin{pmatrix}
1& 1 & \cdots & 1\\
d_1 & d_2 & \cdots &d_n\\
\vdots & \vdots & \ddots & \vdots\\
d_1^{n-1} & d_2^{n-1} & \cdots & d_n^{n-1}
\end{pmatrix}
.
\end{align}
 The numerator $\mathrm{det} (d_i^{n+\kappa_j -j })_{i,j=1}^n$ is an alternating polynomial (in other words it changes sign under any transposition of the variables):
\begin{equation*}
\mathrm{det} (d_i^{n+\kappa_j -j })_{i,j=1}^n
= \sum_{\sigma\in S_n}\epsilon(\sigma)d_{\sigma(1)}^{\kappa_1}\cdots d_{\sigma(n)}^{\kappa_n},
\end{equation*}
where $S_{n}$ is the permutation group of the set $\{1,2,\ldots,n\}$ and $\epsilon(\sigma)$ is the sign of the permutation $\sigma$. 
 
Thus $s_\kappa (d_1,\ldots,d_n)$ is a symmetric function because the numerator and denominator are both alternating, and is a polynomial since all alternating polynomials are divisible by the Vandermonde determinant (see \cite{Fulton,Mac, Sagan} for more details here). For instance, $
s_{(2,1,1)} (x_1, x_2, x_3) = x_1 \, x_2 \, x_3 \, (x_1 + x_2 + x_3) 
$
and
\begin{equation*}
s_{(2,2,0)} (x_1, x_2, x_3)  =  x_1^2 \, x_2^2 + x_1^2 \, x_3^2 + x_2^2 \, x_3^2 + x_1^2 \, x_2 \, x_3 + x_1 \, x_2^2 \, x_3 + x_1 \, x_2 \, x_3^2.  
\end{equation*}
Another related definition is the \emph{Hook length}, $\mathrm{hook}(x)$, of a box $x$ in Young diagram of shape $\kappa$. This is defined as the number of boxes that are in the same row to the right of it plus those boxes in the same column below it, plus one (for the box itself). For instance, in Figure \ref{young}, the hook length of the top-left corner box is $4+2+1=7$. The product of the hook's length of a partition is the product of the hook lengths of all the boxes in the partition. 

Next, we collect a few properties of Schur polynomials $s_\kappa (d_1,\ldots,d_n)$ used in later proofs. For an $n\times n$ matrix $A$ with eigenvalues $\alpha_1,\ldots,\alpha_n$, we use $s_\kappa(A) = s_\kappa (\alpha_1,\ldots,\alpha_n)$. Denote by $(n-k,1^k)$ the partition $(n-k,1,1,\ldots,1)$ with $k$ ones. One of the basic properties of Schur polynomials is that for any integer $l\ge 1$,
\begin{align}\label{eq:prop1}
\text{Tr}(A^l) = \sum_{k=0}^{n-1} (-1)^k s_{(n-l,1^l)} (A).
\end{align}
Let $D_n$ be a diagonal matrix of size $n\times n$. Consider $\Omega_{p,n}$, the Stiefel manifold defined in \eqref{eq:set}, associated with the Haar measure $d\phi$. For any $\Phi \in \Omega_{p,n}$, it is proved in \cite[equation (18)]{FK06} that 
\begin{align}\label{eq:prop2}
\int_{\Omega_{p,n}} s_{\kappa} (\Phi D_n \Phi^*) \, d\phi = \frac{s_{\kappa} (D_n) s_{\kappa} (I_n)}{s_{\kappa} (I_p)}. 
\end{align}

Schur polynomials have a close connection with the border strips of partitions. We follow the definitions in Stanley's book \cite[Chapter 7.17]{Stanley}. A \emph{border strip} is a set of boxes in the Young diagram that forms a contiguous strip and has at most one box on each diagonal. The \emph{height} of a border strip is one less than its number of rows. Given a partition $\lambda\vdash n$ and a decomposition $\rho=(\rho_1,\ldots,\rho_l)$ of $n$. A \emph{border strip tableau} $\chi^{\kappa}(\rho)$ of shape $\kappa$ and type $\rho$ is obtained by replacing each box in the Young diagram of $\kappa$ by one of the integers $\{1,2,\ldots,l\}$ so that the boxes replaced by $i$ form a $\rho_i$ border strip in the diagram which consists of all boxes replaced by  $\{1,2,\ldots,i\}$.

By the celebrated Murnaghan--Nakayama rule (see Corollary 7.17.5 in \cite{Stanley}),  
\begin{equation} \label{eq:MN}
s_{(n-j, 1^j)}(D) = \sum_{\rho=(1^{r_1}, 2^{r_2},\ldots, n^{r_n}) \vdash n} \chi^{(n-j, 1^j)}(\rho) \prod_{l=1}^n \frac{\mathrm{Tr}(D^l)^{r_l}}{l^{r_l} r_l!},
\end{equation} 
where $\chi^{\kappa}(\rho) = \sum_{T} (-1)^{\text{ht}(T)}$ sums over all border-strip tableaux of shape $\kappa$ and type $\rho$. Here $\text{ht} (T)$ is the \emph{height} of a border-strip tableaux (see Section 7.17 in \cite{Stanley} for more details).

\subsection{A new property of the $invcov_p$ estimator}
 Recall $\mathrm{invcov}_p(K)=\mathbb{E} (\Phi^* (\Phi K \Phi^*)^{-1} \Phi)$. We first collect the properties of the $\mathrm{invcov}_p(K)$ estimator obtained in the previous work of Marzetta, Tucci and Simon \cite[Section IV and VI]{MTS}.

\begin{proposition} For a positive semi-definite matrix $K$ of size $m$, one can decompose $K=U D U^*$ where $U$ is unitary and $D=\mathrm{diag}(d_1,\ldots,d_m)$.
\begin{enumerate}\label{prop:invcov}
\item The eigenvectors of $K$ are preserved under the $\mathrm{invcov}_p$ operatoration. More precisely, $\mathrm{invcov}_p(K)=U \mathrm{invcov}_p (D) U^*$ and $\mathrm{invcov}_p (D)$ is diagonal.  

\item The zero-eigenvalues of $K$ are converted to equal positive values. If $D=\text{diag}(D_n, 0_{m-n})$ where $D_n=(d_1,\ldots,d_n)$ is of full rank, then $\mathrm{invcov}_p (D)=\text{diag}(\Lambda_L(D_n), \mu I_{m-n})$ where $\Lambda_L(D_n)=\text{diag}(\lambda_1,\ldots,\lambda_n)$. Besides, for any $1\le k \le n$,
\begin{align}\label{eq:formulaLambba}
\lambda_k=\frac{\partial }{\partial d_k} \int_{\Omega_{p,n}} \mathrm{Tr} (\log (\Phi D_n \Phi^*) )\,d\phi \quad \text{and}\quad \mu= \frac{\det (G)}{\det (\Delta(d_1,\ldots,d_n))}.
\end{align}
Here $\Delta(d_1,\ldots,d_n)$ is the Vandermonde matrix in \eqref{eq:vandermonde} and $G$ is the matrix constructed by replacing the $p$th row of $\Delta(d_1,\ldots,d_n)$ by the row
$$(d_1^{n-(p+1)} \log(d_1),\cdots, d_n^{n-(p+1)} \log(d_n)).$$

\end{enumerate}
\end{proposition}

We prove a new property of the $\mathrm{invcov}_p(K)$ estimator. We will show that $\mathrm{invcov}_p(K)$ has a surprisingly simple algebraic structure despite its rather complicated expression. Assume $K=U D U^*$ where $U$ is unitary and $D=\mathrm{diag}(d_1,\ldots,d_m)$. By Proposition \ref{prop:invcov}, it is enough to study the properties of  $\mathrm{invcov}_p (D)$. 

Let $\mathcal{A}(D)$ be the algebra generated by the matrices $D$ and the $m\times m$ identity matrix $I_{m}$. By the Cayley--Hamilton Theorem, it is clear that 
\begin{equation}
\mathcal{A}(D)=\Big\{ \alpha_{m-1} D^{m-1}+\alpha_{m-2} D^{m-2}+\ldots+\alpha_1 D+\alpha_0 I_m \hspace{0.2cm}:\hspace{0.2cm} \alpha_i \in \mathbb{C} \Big\}.
\end{equation}
We define $\mathcal{D}_m$ as the set of all $m\times m$ diagonal matrices.
\begin{lemma}\label{diag} 
Let $D=\mathrm{diag}(d_1,\ldots,d_m)$ be an $m\times m$ diagonal matrix. If $d_i \neq d_j$ for $i\neq j$ then $\mathcal{A}(D) = \mathcal{D}_m$. If $d_i = d_j$ for some $i \neq j$ then 
$$
\mathcal{A}(D) =\{ \mathrm{diag}(b_1,\ldots,b_i,\ldots,b_i,\ldots,b_m) \,\,:\,\, b_k \in \mathbb{C}\},
$$ 
the set of all diagonal matrices where the $i$th and $j$th entries are equal.
\end{lemma}

\begin{proof}
First assume $d_i\neq d_j$ for all $i\neq j$. It is clear to see $\mathcal{A}(D) \subset \mathcal{D}_m$. On the other hand, for any $B=\text{diag}(b_1,\ldots,b_m) \in \mathcal{D}_m$, we form a system of linear equations,  $$\left( \begin{matrix} b_1 \\ \vdots \\ b_m \end{matrix} \right) =\left( \begin{matrix} 1 & d_1 & d_1^2 & \ldots &d_1^{m-1} \\ \vdots& &\ldots & &\vdots \\ 1 & d_m & d_m^2 & \ldots & d_m^{m-1} \end{matrix} \right) \left( \begin{matrix} \alpha_0 \\ \vdots \\ \alpha_{m-1} \end{matrix} \right) := V \left( \begin{matrix} \alpha_0 \\ \vdots \\ \alpha_{m-1} \end{matrix} \right).$$
The matrix $V$ is a Vandermonde matrix with $\text{det}(V) = \prod_{i<j} (d_i -d_j)$. The matrix $V$ is invertible by our assumption. Thus we can find a vector $(\alpha_0, \ldots, \alpha_{m-1})$ such that $$B=\alpha_0 I_m+ \alpha_1 D+\ldots + \alpha_{m-1} D^{m-1} \in \mathcal{A}(D).$$
This completes the proof. To prove the second part we use essentially the same approach as before.
\end{proof}

\begin{theorem}
The matrix $\mathrm{invcov}_p (D)$ belongs to the algebra $\mathcal{A}(D).$
\end{theorem}

\begin{proof}
By Proposition \ref{prop:invcov}, if the matrix $D$ is equal to $D=\text{diag}(D_n, 0_{m-n})$ where $D_n=(d_1,\ldots,d_n)$ is of full rank, then $\mathrm{invcov}_p (D)=\text{diag}(\Lambda_L(D_n), \mu I_{m-n})$ where $\Lambda_L(D_n)=\text{diag}(\lambda_1,\ldots,\lambda_n)$. And
$$
\lambda_k=\frac{\partial F(d_1,\ldots,d_n)}{\partial d_k},
$$ 
where we define $F(d_1,\ldots,d_n):= \int_{\Omega_{p,n}} \mathrm{Tr}( \log (\Phi D_n \Phi^*) )d\phi$ for brevity. Recall $\Phi \in \Omega_{p,n}$ defined in \eqref{eq:set}. By \eqref{eq:prop1} and \eqref{eq:prop2}, for any integer $l \ge 1$
$$
\int_{\Omega_{p,n}} \mathrm{Tr}\big( (\Phi D_n \Phi^*) \big)^l \,d\phi = \sum_{k=0}^{p-1} (-1)^k c_k^{(n,p)} s_{(l-k,1^k)}(D_n),
$$ 
where $s_{(l-k,1^k)}(D_n)$ are the Schur polynomials and $c_k^{(n,p)}$ are  explicit constants (see (78) in \cite{MTS}). From Lemma \ref{diag}, it is enough to show that if $d_i=d_j$ for some $i \neq j$, then $\lambda_i =\lambda_j$.
 By linearity and continuity, $F(d_1,\ldots,d_n)$ is symmetric. Hence assuming $d_i=d_j$, ${\partial F}/{\partial d_i}={\partial F}/{\partial d_j}$, which implies $\lambda_i =\lambda_j$. This completes the proof.
\end{proof}

\subsection{Formulas for computing $\mathbb{E}(\Phi^* (\Phi D_n \Phi^*)^l \Phi)$.}\label{sec:formula}

In order to obtain the explicit formulas of $\mathrm{cov}_p$ and $\mathrm{invcov}_p$ in \cite{MTS}, it involves computing $\mathbb{E}(\Phi^* f(\Phi D_n \Phi^*) \Phi)$ for a differentiable function $f(x)$ (see parts A and B in section VI in \cite{MTS}) and a diagonal matrix $D_n=\text{diag}(d_1,\ldots,d_n)$ with all $d_i$'s positive. For instance, \cite[Lemma 1]{MTS} asserts that if $f$ is differentiable on the interval $[\min\{d_i\}, \max\{d_i\}]$, then 
$$\frac{\partial }{\partial d_k} \int_{\Omega_{p,n}} \mathrm{Tr} (f (\Phi D_n \Phi^*) )\, d\phi = \Big(\int_{\Omega_{p,n}} \Phi^* f' (\Phi D_n \Phi^*) \Phi \, d\phi  \Big)_{kk}= \mathbb{E} \big( \Phi^* f' (\Phi D_n \Phi^*) \Phi \big)_{kk}.$$
Note the eigenvalue $\lambda_k$ of $\mathrm{invcov}_p(D)$ given in \eqref{eq:formulaLambba} is the left hand side of above identity with $f(x)=\log x$. To  further understand the $\mathrm{invcov}_p$ operator, it is helpful to have the explicit formula for the eigenvalues $\lambda_k$'s. By continuity and linearity, it is enough to provide formulas for computing $\mathbb{E}(\Phi^* (\Phi D_n \Phi^*)^l \Phi)$. In this subsection, we derive such formulas.

First, we observe that $\mathbb{E}(\Phi^* (\Phi D_n \Phi^*)^l \Phi)$ is still a diagonal matrix. The idea of proof is exactly the same as the proof of Proposition \ref{prop:invcov}. We recall a fact that a matrix $A$ is diagonal if and only if $\Omega^* A \Omega = A$ for any diagonal unitary matrix $\Omega$. Note that
\begin{align*}
\Omega^*\mathbb{E}(\Phi^* (\Phi D_n \Phi^*)^l \Phi) \Omega = \E \Big( (\Phi \Omega)^* \big( \Phi \Omega (\Omega^* D_n \Omega)(\Phi \Omega)^* \big)^l \Phi \Omega \Big) = \mathbb{E}(\Phi^* (\Phi D_n \Phi^*)^l \Phi),
\end{align*}
where we use that $ \Phi \Omega$ has the same distribution as $\Omega$, and $\Omega^* D_n \Omega= D_n$.

To compute the diagonal entries of $\mathbb{E}(\Phi^* (\Phi D_n \Phi^*)^l \Phi)$, using Lemma 1 in \cite{MTS}, we have 
\begin{align}\label{eq:for}
\big( \mathbb{E}(\Phi^* (\Phi D_n \Phi^*)^l \Phi) \big)_{ii}= \Big( \int_{\Omega_{p,n}} \Phi^* (\Phi D_n \Phi^*)^l \Phi \, d\phi  \Big)_{ii} = \frac{\partial}{\partial d_i} \int_{\Omega_{p,n}} \frac{1}{l+1} \mathrm{Tr} (\Phi D_n \Phi^*)^{l+1} \, d \phi.
\end{align}
Denote $N=l+1$ for convenience. By \eqref{eq:prop1} and \eqref{eq:prop2}, we see that
\begin{align}\label{eq:int}
\int_{\Omega_{p,n}} \mathrm{Tr} \big( (\Phi D_n \Phi^*)^{N} \big) \, d \phi &= \sum_{j=0}^{p-1} (-1)^j \frac{s_{(N-j,1^j)} (I_p)}{s_{(N-j,1^j)} (I_n)} s_{(N-j,1^j)}(D_n) \nonumber\\
&=\sum_{j=0}^{p-1} (-1)^j \frac{(N+p-(j+1))! (n-(j+1))!}{(N+n-(j+1))! (p-(j+1))!} s_{(N-j,1^j)}(D_n).
\end{align}

Using the formula \eqref{eq:MN}, one has
\begin{equation}
\begin{split}
\frac{\partial s_{(N-j, 1^j)}(D_n)}{\partial d_i} &= \sum_{k=1}^N d_i^{k-1} \Big( \sum_{\rho=(1^{r_1}, 2^{r_2},\ldots, N^{r_N}) \vdash N} \chi^{(N-j, 1^j)}(\rho)\frac{r_k \mathrm{Tr}(D^k)^{r_k-1} }{k^{r_k-1} r_k!} \prod_{l\neq k} \frac{\mathrm{Tr}(D^l)^{r_l}}{l^{r_l} r_l!} \Big)\\
&:=\sum_{k=1}^N d_i^{k-1}\cdot \tilde{c}_{k-1}=\sum_{k=0}^{N-1} \tilde{c}_{k}d_i^{k}.
\end{split}
\end{equation}

Therefore, combining \eqref{eq:for} and \eqref{eq:int}, we obtain
\begin{equation*}
\begin{split}
\big( \mathbb{E}(\Phi (\Phi^* D \Phi)^l \Phi^*) \big)_{ii}&= \frac{1}{l+1} \sum_{j=0}^{p-1} (-1)^j \frac{(l+1+p-(j+1))! (n-(j+1))!}{(l+1+n(j+1))! (p-(j+1))!} \frac{\partial s_{(N-j,1^j)}(D_n)}{\partial d_i}\\
&=\sum_{k=0}^l \Big( \frac{\tilde{c}_k}{l+1} \sum_{j=0}^{p-1} (-1)^j \frac{(l+p-j)! (n-j-1))!}{(l+n-j)! (p-j-1)!}  \Big) d_i^k :=\sum_{k=0}^l a_k d_i^k.
\end{split}
\end{equation*}

The coefficients $a_k$ depend only on $D_n,p$ and $l$. Thus we are able to show $\mathbb{E}(\Phi^* (\Phi D_n \Phi^*)^l \Phi)$ is a polynomial in $D_n$ of degree $l$,
$$
\mathbb{E} \big(\Phi^* (\Phi D_n \Phi^*)^l \Phi \big)=\sum_{k=0}^l a_k D_n^k
$$
where the coefficients are
\begin{align*}
a_k&=\frac{1}{l+1} \Big(\sum_{j=0}^{p-1} (-1)^j \frac{(l+p-j)! (n-j-1))!}{(l+n-j)! (p-j-1)!}\Big) \\
&\quad\quad\quad\quad\quad \cdot \Big( \sum_{\rho=(1^{r_1}, \ldots, (l+1)^{r_{l+1}}) \vdash l+1} \chi^{(l+1-j, 1^j)}(\rho)\frac{r_{k+1} \mathrm{Tr}(D^{k+1})^{r_{k+1}-1} }{{(k+1)}^{r_{k+1}-1} r_{k+1}!} \prod_{l\neq {k+1}} \frac{\mathrm{Tr}(D^l)^{r_l}}{l^{r_l} r_l!} \Big). 
\end{align*}

In the Appendix \ref{sec:small}, we provide a small dimensional example to show how to apply the derived formula for computation.

\section{The mean conjugate estimator under Ewens measure}\label{sec:ewens}

Let $S_m$ be the set of permutations of the set $[m]:=\{1,2,\ldots,m\}$. For each permutation $\sigma\in S_m$, by cycle decomposition, $\sigma$ can be viewed as the disjoint union of cycles of varying lengths. The Ewens measure is a probability measure on the set of permutations that depends on a parameter $\theta>0$. In this measure, each permutation has a weight proportional to its total number of cycles.  More specifically, for each permutation $\sigma$ in $S_m$ its probability is equal to 
\begin{align}\label{eq:ewens}
p_{\theta, m}(\sigma) = \frac{\theta^{\#\text{cycl}(\sigma)}}{\theta(\theta+1)\ldots(\theta+m-1)},
\end{align}
where $\theta >0$ and $\#\text{cycl}(\sigma)$ is the number of cycles in $\sigma$. The case $\theta=1$ corresponds to the uniform measure. This measure has recently appeared in mathematical physics models (see e.g. \cite{3} and \cite{9}) and one has only recently started to gain insight into the cycle structures of such random permutations.

Let $\sigma$ be a permutation in $S_m$, the corresponding permutation matrix $M_\sigma$ is the $m \times m$ matrix defined as $M_\sigma(i,j)=\mathbf{1}_{\sigma(i)}(j).$ If we denote $e_i$ to be a $1 \times m$ vector such that the $i$--th entry is equal to $1$ and all the others entries are $0$, then 
$$
M_\sigma=\left(\begin{matrix} e_{\sigma(1)}\\ \vdots\\ e_{\sigma(m)} \end{matrix} \right),
$$ 
which is, of course, a unitary matrix. Given the sample covariance matrix $K$ we define the new estimator for $\Sigma$ as
\begin{equation}\label{eq:K}
K_\theta := \mathbb{E} (M_{\sigma} K M_{\sigma}^*),
\end{equation}
where the expectation is taken with respect to the Ewens measure of parameter $\theta$.

\begin{theorem} \label{permutation}
Let $K=(a_{ij})$ be an $m \times m$ matrix in $\mathbb{C}^{m\times m}$. Then
$K_\theta= \mathbb{E} (M_{\sigma} K M_{\sigma}^*)$ is an $m \times m$ matrix such that the diagonal terms satisfy
\begin{equation}
(K_\theta)_{ii}=\frac{\theta-1}{\theta+m-1}a_{ii}+\frac{1}{\theta+m-1}\mathrm{Tr}(K),
\end{equation}
and the non--diagonal terms $(i \neq j)$ satisfy
\begin{equation}
\begin{split}
(K_\theta)_{ij}&=\frac{1}{(\theta+m-2)(\theta+m-1)} \Big( \theta^2 a_{ij} + (\theta-1) a_{ji} + \theta \sum_{k \neq i,j} (a_{ik}+ a_{kj} ) + \sum_{l \neq i, k \neq j \atop k\neq l} a_{lk}  \Big)\\
&= \frac{1}{(\theta+m-2)(\theta+m-1)} \Big( (\theta^2-1) a_{ij} + (\theta-1) a_{ji} + (\theta-1) \sum_{k \neq i,j} (a_{ik}+ a_{kj} ) + \sum_{l \neq k} a_{lk} \Big).
\end{split}
\end{equation}
\end{theorem}

\begin{remark}\label{rem:uniform}
If $\theta=1$, then 
\begin{align}\label{eq:unf}
K_1=\alpha \frac{\mathbf{e} \mathbf{e}^T}{m} + \beta (I_m- \frac{\mathbf{e} \mathbf{e}^T}{m})\hspace{0.2cm}\text{where}\hspace{0.2cm} \alpha=\frac{\mathbf{e}^T K\mathbf{e}}{m}=\frac{\sum_{i,j}a_{ij}}{m}~ \text{,}~ \beta= \frac{\mathrm{Tr}(K)-\alpha}{m-1}
\end{align} 
and $\mathbf{e}=(1,1,\ldots,1)^T$. This result already been shown in Prop. 2.2 of \cite{Viana}. 
\end{remark}

\begin{remark}\label{rem:diag}
If $K=D=\text{diag}(d_1,\ldots,d_m)$, then 
$$
K_\theta= \frac{\theta-1}{\theta+m-1} D+ \frac{\mathrm{Tr}(D)}{\theta+m-1} I_m,
$$ 
which corresponds to the diagonal loading.
\end{remark}

\begin{proof} First, $$M_\sigma K M^* = \left(\begin{matrix} e_{\sigma(1)} \\ \vdots \\ e_{\sigma(m)} \end{matrix} \right) K \big( \begin{matrix} e^*_{\sigma(1)} & \cdots & e^*_{\sigma(m)} \end{matrix} \big) = \Big(\sum_{l=1}^m \sum_{k=1}^m a_{kl} e_{\sigma(i)}(k) e_{\sigma(j)}(l) \Big) = (a_{\sigma(i) \sigma(j)})_{1\le i, j \le m}.$$

For diagonal terms, recall the probability measure $p_{\theta, m}$ in \eqref{eq:ewens},
\begin{equation*}
\begin{split}
(K_\theta)_{ii}&=\big(\mathbb{E} (M_{\sigma} K M_{\sigma}^*)\big)_{ii}= \sum_{\sigma \in S_m} p_{\theta, m}(\sigma) a_{\sigma(i) \sigma(i)} = a_{ii} \sum_{\sigma \in S_m \atop \sigma(i) =i } p_{\theta, m}(\sigma)+ \sum_{l \neq i} a_{ll} \sum_{\sigma \in S_m \atop \sigma(i) =l} p_{\theta, m}(\sigma)\\
&= a_{ii} \frac{\theta}{\theta+m-1}  \sum_{\tilde\sigma \in S_{m-1}} p_{\theta,m-1}(\tilde\sigma) + \sum_{l\neq i} \frac{a_{ll}}{\theta+m-1} \sum_{\hat\sigma(l) }  p_{\theta,m-1}(\hat\sigma(l)) \\
&= \frac{\theta}{\theta+m-1} a_{ii} + \frac{1}{\theta+m-1} \sum_{l \neq i}  a_{ll}=\frac{\theta-1}{\theta+m-1}a_{ii}+ \frac{1}{\theta+m-1} \mathrm{Tr}(K).
\end{split}
\end{equation*}

Now we compute the off--diagonal terms $(K_\theta)_{ij} ~(i \neq j)$. For $\sigma \in S_m$, if $\sigma(i)=i$ and $\sigma(j)=j$ then $\sigma=(i)(j)\sigma_1$ with $\sigma_1 \in S_{m-2}$, $\#\text{cycl}(\sigma)=\#\text{cycl}(\sigma_1)+2$ and 
$$
p_{\theta,m}(\sigma)=\frac{\theta^2}{(\theta+m-2)(\theta+m-1)} p_{\theta,m-2}(\sigma_1).
$$
If $\sigma(i)=j$ and $\sigma(j)=i$ we erase $i$ and $j$ from $\sigma$ to obtain $\sigma_2 \in S_{m-2}$, and 
$$
p_{\theta, m}(\sigma)= \frac{\theta}{(\theta+m-2)(\theta+m-1)} p_{\theta, m-2}(\sigma_2).
$$
If $\sigma(i)=i$ and $\sigma(j)=k\neq i,j$ then $\sigma=(i)\hat\sigma$ with $\hat\sigma \in S_{m-1}$ and $\#\text{cycl}(\sigma)=\#\text{cycl}(\hat\sigma)+1$. Furthermore, we can erase $j$ from $\hat\sigma$ to get a new permutation $\sigma_3(k) \in S_{m-2}$ such that $\#\text{cycl}(\sigma_3(k))=\#\text{cycl}(\hat\sigma)$ and finally 
$$
p_{\theta, m}(\sigma)= \frac{\theta}{(\theta+m-2)(\theta+m-1)} p_{\theta, m-2}(\sigma_3(k)).
$$ 
Notice that $\sum_{\sigma_3(k)}p_{\theta, m-2}(\sigma_3(k))=1$.

If $\sigma(i)=l\neq i,j$ and $\sigma(j)=j$ then as above we can have $\sigma_4(l) \in S_{m-2}$ such that 
$$
p_{\theta, m}(\sigma)= \frac{\theta}{(\theta+m-2)(\theta+m-1)} p_{\theta, m-2}(\sigma_4(l))
$$ 
and again $\sum_{\sigma_4(l)}p_{\theta, m-2}(\sigma_4(l))=1.$

If $\sigma(i)=l \neq i$ and $\sigma(j) =k \neq j$ ($k\neq l$) we exclude the case that $\sigma(i)=j, \sigma(j)=i$ and we erase $i$ and $j$ from $\sigma$ to obtain $\sigma_5(l,k) \in S_{m-2}$. Thus 
$$
p_{\theta, m}(\sigma)= \frac{1}{(\theta+m-2)(\theta+m-1)} p_{\theta, m-2}(\sigma_5(l,k))
$$
and $\sum_{\sigma_5(l,k)}p_{\theta, m-2}(\sigma_5(l,k))=1.$ 

Therefore, for $i \neq j$ 
\begin{equation*}
\begin{split}
&(K_\theta)_{ij} = \sum_{\sigma \in S_m } p_{\sigma, m}(\sigma) a_{\sigma(i) \sigma(j)} \\
&= a_{ij}\frac{\theta^2}{(\theta+m-2)(\theta+m-1)} \sum_{\sigma_1 \in S_{m-2} }  p_{\theta,m-2}(\sigma_1)\\
& + a_{ji}\frac{\theta}{(\theta+m-2)(\theta+m-1)} \sum_{\sigma_2 \in S_{m-2}}  p_{\theta,m-2}(\sigma_2) \\
&+  \sum_{k \neq i,j} a_{ik} \frac{\theta}{(\theta+m-2)(\theta+m-1)} \sum_{\sigma_3(k) \in S_{m-2}} p_{\theta, m-2}(\sigma_3(k))  \\
&+ \sum_{l \neq i,j} a_{lj} \frac{\theta}{(\theta+m-2)(\theta+m-1)} \sum_{\sigma_4(l) \in S_{m-2}} p_{\theta, m-2}(\sigma_4(l)) \\ 
&+ \sum_{\substack{k \neq i,j \text{ and } l \neq i,j\\ k \neq l}}  \sum_{\sigma_5(k,l) \in S_{m-2} } a_{lk} \frac{1}{(\theta+m-2)(\theta+m-1)} p_{\theta, m-2}(\sigma_5(k,l))\\
&= \frac{1}{(\theta+m-2)(\theta+m-1)} \Big( \theta^2 a_{ij} + (\theta-1) a_{ji} + \theta \sum_{k \neq i,j} (a_{ik}+ a_{kj} ) + \sum_{\substack{k \neq i,j \text{ and } l \neq i,j\\ k \neq l}}  a_{lk}  \Big).
\end{split}
\end{equation*}
\end{proof}

\section{Hybrid Method}\label{sec:mix}

In this section, we combine the ideas of the first two methods to create a third hybrid method. First, we extend the definition of a permutation. For an integer $p\le m$, let 
$$
S_{p,m}:=\Big\{ \sigma~ : \sigma~ \text{an injection from} ~\{1,2,\ldots,p\}~ \text{to}~ \{1,2,\ldots m\} \Big\}.
$$
The size of the set $S_{p,m}$ is $\frac{m!}{(m-p)!}$ and it is clear that $S_{m,m}$ is the set of all permutations on $[m]$. For $\sigma \in S_{p,m}$, the associated $p\times m$ matrix takes the form
$$
V_{\sigma}:=\left( \begin{matrix} e_{\sigma(1)} \\ e_{\sigma(2)} \\ \vdots \\e_{\sigma(p)}  \end{matrix} \right),
$$
where $e_{\sigma(i)}=(e_{\sigma(i)}^1,e_{\sigma(i)}^2,\ldots,e_{\sigma(i)}^m)$ is a $1\times m$ row vector with the $\sigma(i)$--th entry 1 and all others 0. Notice
\begin{equation}
V_{\sigma} V_{\sigma}^{\text{T}} = I_{p}, 
\end{equation} 
and 
\begin{equation} \label{eq:projection} 
P_{\sigma}:=V_{\sigma}^{\text{T}} V_{\sigma} =\mathrm{diag}(b^{\sigma}_1,\ldots,b^{\sigma}_m),
\end{equation}
where
$$
   b^{\sigma}_i = \sum_{l=1}^p (e_{\sigma(l)}(i))^2 =\left\{
     \begin{array}{lr}
       1  \,\,\, \mathrm{if}~ i \in \{\sigma(1),\ldots,\sigma(p)\},\\
       0  \,\,\, \mathrm{otherwise}.
     \end{array}
     \right.
$$
Next, we use the \emph{Ewens measure} on the permutation sets to define a probability on the set $S_{p,m}$. For each $\sigma \in S_{p,m}$, consider the set 
$$
\Omega_{\sigma} := \Big\{ \tilde\sigma \in S_m~:~ \tilde\sigma  _{\{1,\ldots,p \} }=\sigma \Big\}.
$$ 
In other words, $\Omega_{\sigma}$ is the set of all permutations in $S_{m}$ whose restriction to the set $\{1,2,\ldots, p\}$ is equal to $\sigma$. Recall that $p_{\theta,m}$ is the \emph{Ewens measure} on $S_m$ with parameter $\theta$. Define the probability measure on $S_{p,m}$ for $\sigma\in S_{p,m}$ as
\begin{equation}
\mu_{\theta,m,p}(\sigma):=p_{\theta,m}(\Omega_{\sigma})=\sum_{\tilde{\sigma} \in \Omega_{\sigma}} p_{\theta,m}(\tilde\sigma).
\end{equation}

Now we are ready to introduce two new operators
\begin{equation}\label{eq:ewens1}
K_{\theta,m,p}:= \mathbb{E} \Big( V_{\sigma}^T (V_{\sigma} K V_{\sigma}^T) V_{\sigma} \Big)
\end{equation} 
\begin{equation}\label{eq:ewens2} 
\tilde{K}_{\theta,m,p}:=\mathbb{E}\Big(V_\sigma^T(V_\sigma K V_\sigma^T)^{+}V_\sigma\Big),
\end{equation}
where $(V_\sigma K V_\sigma^T)^{+}$ is the Moore--Penrose pseudo inverse of the matrix $V_\sigma K V_\sigma^T$. Recall the Moore--Penrose pseudo inverse of a square matrix $A$ is a matrix $A^+$ of the same size and satisfies $AA^+$ and $A^+ A$ are both Hermitian, $AA^+A=A$ and $A^+ A A^+=A$. We use $K_{\theta,m,p}$ as an estimate for $\Sigma$ and $\tilde{K}_{\theta,m,p}$ for $\Sigma^{-1}$. Now we show a few results on these new estimators.

\begin{theorem}
Let $K=(a_{ij})$ be an $m\times m$ complex matrix. Then $K_{\theta,m,p}$ as in \eqref{eq:ewens1} is an $m\times m $ matrix such that the diagonal entries are equal to
$$
(K_{\theta,m,p})_{ii} =\left\{
     \begin{array}{lr}
       \frac{\theta+p-1}{\theta+m-1} a_{ii}, ~\mathrm{if}~1\le i\le p, \\
          \\
       \frac{p}{\theta+m-1} a_{ii}, ~\mathrm{if}~p+1\le i\le m,
     \end{array}
     \right.
$$
and the non--diagonal entries, assuming $i<j$ (if $j<i$ then exchange $i$ and $j$ in the following expression) are equal to 
$$
(K_{\theta,m,p})_{ij} =\left\{
     \begin{array}{lr}
       \frac{(\theta+p-1)(\theta+p-2)}{(\theta+m-1)(\theta+m-2)} a_{ij}   , ~\mathrm{if}~1\le i < j \le p ,\\
          \\
       \frac{(p-1)(\theta+p-1)}{(\theta+m-1)(\theta+m-2)} a_{ij}  , ~\mathrm{if}~1\le i\le p <j \le m,\\
       \\
       \frac{p(p-1)}{(\theta+m-1)(\theta+m-2)} a_{ij}  , ~\mathrm{if}~p< i< j \le m.
     \end{array}
     \right.
$$
\end{theorem}

\begin{remark}
In the special case that $K=\mathrm{diag}(d_1,\ldots,d_m)$ is a diagonal matrix , then 
$$
K_{\theta,m,p}= \frac{p}{\theta+m-1} K + \frac{\theta-1}{\theta+m-1} \mathrm{diag}(d_1,\ldots,d_p, 0, \ldots,0).
$$
For instance, if $p=1$ and $m=3$ then 
$$
K_{\theta,3,1} =\frac{1}{\theta+2} \mathrm{diag} (\theta a_{11}, a_{22}, a_{33}).
$$
\end{remark}

\begin{remark}
In the general case with $p=2$ and $m=3$ then 
$$
K_{\theta,3,2} =\frac{1}{\theta+2} \left( \begin{matrix} (\theta+1) a_{11} & \theta a_{12} & a_{13} \\ \theta a_{21} & (\theta+1) a_{22} & a_{23} \\ a_{31} & a_{32} & 2a_{33} \end{matrix} \right).
$$
\end{remark}

\begin{proof}
Recall from Equation \eqref{eq:projection} that 
$$
P_{\sigma}=V_{\sigma}^{\text{T}} V_{\sigma} =\mathrm{diag}(b_1^{\sigma},\ldots,b_m^{\sigma}),
$$ 
thus $ V_{\sigma}^T (V_{\sigma} K V_{\sigma}^T) V_{\sigma}  = (b_i^{\sigma} b_j^{\sigma} a_{ij})_{1\le i, j \le m},$ where
$$b^{\sigma}_i = \sum_{l=1}^p (e_{\sigma(l)}(i))^2 =\left\{
     \begin{array}{lr}
       1  \,\,\, \mathrm{if}~ i \in \{\sigma(1),\ldots,\sigma(p)\},\\
       0  \,\,\, \mathrm{otherwise}.
     \end{array}
     \right.$$
For the diagonal entries, if $1\le i\le p$, 
\begin{equation*}
\begin{split}
(K_{\theta,m,p})_{ii} &= \sum_{\sigma\in S_{m,p}} \mu_{\theta,m,p}(\sigma) (b_i^{\sigma})^2 a_{ii}= a_{ii} \sum_{l=1}^p \sum_{\sigma \in S_{m,p}, \sigma(l)=i} \mu_{\theta, m, p}(\sigma) \\
&= a_{ii} \Big(\sum_{\sigma \in S_{m,p}, \sigma(i)=i} \mu_{\theta,m,p} + \sum_{l\neq i} \sum_{\sigma \in S_{m,p}, \sigma(l)=i} 
\mu_{\theta,m,p} \Big)\\
&= a_{ii} \Big( \frac{\theta}{\theta+m-1} \sum_{\sigma' \in S_{m-1,p-1}} \mu_{\theta,m-1,p-1}+ \frac{p-1}{\theta+m-1} \sum_{\sigma' \in S_{m-1,p-1}} \mu_{\theta,m-1,p-1}\Big)\\
&= \frac{\theta+p-1}{\theta+m-1} a_{ii}.
\end{split}
\end{equation*}
If $p+1\le i\le m$, 
\begin{equation*}
\begin{split}
(K_{\theta,m,p})_{ii} &= \sum_{\sigma\in S_{m,p}} \mu_{\theta,m,p}(\sigma) (b_i^{\sigma})^2 a_{ii}= a_{ii} \sum_{l=1}^p \sum_{\sigma \in S_{m,p}, \sigma(l)=i} \mu_{\theta, m, p}(\sigma) \\
&= a_{ii} \Big(\frac{p}{\theta+m-1} \sum_{\sigma' \in S_{m-1,p-1}} \mu_{\theta,m-1,p-1} \Big)\\
&= \frac{p}{\theta+m-1} a_{ii}.
\end{split}
\end{equation*}
For non-diagonal entries, if $1\le i < j \le p$, which turns out to be the most complicated case, $b_i^{\sigma} b_j^{\sigma} a_{ij}$ is non zero if $i, j \in \{\sigma(1), \ldots, \sigma(p) \}$. Thus $$
(K_{\theta,m,p})_{ij} =a_{ij} \sum_{s,t \in [p], s\neq t} \sum_{\sigma \in S_{m,p},\atop \sigma(s)=i, \sigma(t)=j} \mu_{\theta,m,p} (\sigma).
$$ 
We divide the previous sum into five parts: 
\begin{enumerate}
\item  If $\sigma(i)=i$ and $\sigma(j)=j$ we ``erase" $i$ and $j$ from the sets $[p]$ and $[m]$ to get a new injection $\sigma_1$ from $[p]\backslash \{i,j \}$ to $[m]\backslash \{i,j \}$ with $\#\text{cycl}(\sigma)=\#\text{cycl}(\sigma_1)+2$;

\vspace{0.2cm}
\item If $\sigma(s)=i$ for some $s\in [p]\backslash \{i,j\}$ and $\sigma(j)=j$ we ``erase" $j$ from the sets $[p]$ and $[m]$ and consider $s$ and $i$ as one number $\tilde s$. Then we get a new injection $\sigma_2 : [p]\cup {\tilde s} \backslash \{i,j,s \} \rightarrow [m]\cup {\tilde s} \backslash \{i,j,s \}$ with $\#\text{cycl}(\sigma)=\#\text{cycl}(\sigma_2) +1$; 

\vspace{0.2cm}
\item  If $\sigma(t)=j$ for some $t\in [p]\backslash \{i,j\}$ and $\sigma(i)=i$ then, similarly to case (2), by exchanging the roles of $i$ and $j$ we can get a new injection $\sigma_3$ with $\#\text{cycl}(\sigma)=\#\text{cycl}(\sigma_3) +1$;

\vspace{0.2cm}
\item  If $\sigma(s)=i$ and  $\sigma(t)=j$ with $s\neq t$ for some $s \in [p]\backslash \{i \}$ and $t \in [p] \backslash \{j\}$ then we consider $s$ and $i$ as a new number $\tilde s$ and $t$ and $j$ as a new number $\tilde t$ to get a new injection $\sigma_4 : [p]\cup {\tilde s, \tilde t} \backslash \{i,j,s,t \} \rightarrow [m]\cup {\tilde s,\tilde t} \backslash \{i,j,s,t \}$ with $\#\text{cycl}(\sigma)=\#\text{cycl}(\sigma_4)$;

\vspace{0.2cm}
\item  If $\sigma(i)=j$ and $\sigma(j)=i$ we ``erase" $i$ and $j$ to get a new injection $\sigma_5 :[p]\backslash \{i,j \} \rightarrow [m]\backslash \{i,j\}$ with $\#\text{cycl}(\sigma)=\#\text{cycl}(\sigma_5)+1$.
\end{enumerate}
\begin{equation*}
\begin{split}
(K_{\theta,m,p})_{ij} &= a_{ij} \frac{ \theta^2}{(\theta+m-1)(\theta+m-2)} \sum_{\sigma_1 \in S_{m-2,p-2}} \mu_{\theta,m-2,p-2} (\sigma_1)\\
& + \frac{a_{ij} \theta(p-2)}{(\theta+m-1)(\theta+m-2)}\sum_{\sigma_2 \in S_{m-2,p-2}} \mu_{\theta,m-2,p-2} (\sigma_2)\\ 
& +\frac{a_{ij}  \theta(p-2)}{(\theta+m-1)(\theta+m-2)}\sum_{\sigma_3 \in S_{m-2,p-2}} \mu_{\theta,m-2,p-2} (\sigma_3) \\
&+ a_{ij}\frac{(p-2)^2 +(p-2)}{(\theta+m-1)(\theta+m-2)}\sum_{\sigma_4 \in S_{m-2,p-2}} \mu_{\theta,m-2,p-2} (\sigma_4) \\
& +  \frac{a_{ij}\theta}{(\theta+m-1)(\theta+m-2)}\sum_{\sigma_5 \in S_{m-2,p-2}} \mu_{\theta,m-2,p-2} (\sigma_5)\\
&= \frac{(\theta+p-1)(\theta+p-2)}{(\theta+m-1)(\theta+m-2)} a_{ij}.
\end{split}
\end{equation*}

For $1\le i \le p <j \le m$  we only need consider two cases: $s=i$ and $s\neq i$, 
\begin{equation*}
\begin{split}
(K_{\theta,m,p})_{ij} & = a_{ij} \frac{ \theta (p-1)}{(\theta+m-1)(\theta+m-2)} \sum_{\sigma_1 \in S_{m-2,p-2}} \mu_{\theta,m-2,p-2} (\sigma_1)\\
&+ a_{ij}\frac{(p-1)^2}{(\theta+m-1)(\theta+m-2)} \sum_{\sigma_2 \in S_{m-2,p-2}} \mu_{\theta,m-2,p-2} (\sigma_2) \\
& = a_{ij} \frac{(p-1)(p+\theta-1)}{(\theta+m-1)(\theta+m-2)}.
\end{split}
\end{equation*}
For $p <i <j \le m$, $$(K_{\theta,m,p})_{ij} = a_{ij} \frac{p(p-1)}{(\theta+m-1)(\theta+m-2)}.$$
\end{proof}

Now we consider the estimate $\tilde K_{\theta,m,p}$ as in Equation \eqref{eq:ewens2}. First we analyze the case when $K$ is diagonal.
\begin{theorem}Let $D=D_m=\mathrm{diag} (d_1, \ldots,d_n, 0,\ldots, 0)$, then for $p\le n$,
$$
\tilde K_{\theta,m,p} = \mathbb{E}\Big(V_\sigma^T(V_\sigma D V_\sigma^T)^{+}V_\sigma\Big) = \frac{\theta+p-1}{\theta+m-1} D^{+} -\frac{\theta-1}{\theta+m-1} \mathrm{diag} (d_1^{-1}, \ldots, d_p^{-1}, 0, \ldots, 0),
$$
where $D^+ = \mathrm{diag} (d_1^{-1}, \ldots, d_n^{-1}, 0, \ldots, 0)$ by definition.
\end{theorem} 

\begin{proof}
First we notice that $W_\sigma:= V_\sigma D V_\sigma^T = (\sum_{i=1}^n d_l e_{\sigma(i)}(l) e_{\sigma(j)}(l))_{1\le i,j \le p}$ is a diagonal matrix. 
For $1\le i \le p$, 
$$
(W_\sigma)_{ii} = \sum_{l=1}^n d_l (e_{\sigma(i)}(l))^2 =\left\{
     \begin{array}{lr}
       d_{\sigma(i)} \,\,\, ~\mathrm{if}~ \sigma(i) \in [n],\\
       0  \,\,\, ~\mathrm{otherwise}.
     \end{array}
     \right. 
$$
Thus 
$$
W_{\sigma}= \mathrm{diag} (d_{\sigma(1)} \mathbf{1}_{\sigma(1) \in [n]}, \ldots, d_{\sigma(p)} \mathbf{1}_{\sigma(p) \in [n]} )$$ and $$W_\sigma^+ = \mathrm{diag} \left( (d_{\sigma(1)} \mathbf{1}_{\sigma(1) \in [n]})^{+}, \ldots, (d_{\sigma(p)} \mathbf{1}_{\sigma(p) \in [n]})^+ \right).
$$
Next $V_{\sigma}^T W^+ V_{\sigma} =\sum_{l=1}^p  (d_{\sigma(l)} \mathbf{1}_{\sigma(l) \in [n]})^{+}$ is still a diagonal matrix where for $1\le i \le m$
$$
(V_{\sigma}^T W^+ V_{\sigma})_{ii} =\left\{
     \begin{array}{lr}
       (d_{\sigma(l)} \mathbf{1}_{\sigma(l) \in [n]})^{+} \,\,\, ~\mathrm{if}~ i \in \{\sigma(1), \ldots, \sigma(p) \},\\
       0 \,\,\, ~\mathrm{otherwise}.
     \end{array}
     \right. 
$$
Therefore $\tilde K_{\theta,m,p}$ is also diagonal and 
$$
(\tilde K_{\theta,m,p})_{ii} = \sum_{l=1}^p \sum_{\sigma \in S_{m,p},\atop \sigma(l) =i} \mu_{\theta,m,p}(\sigma)(d_i \mathbf{1}_{i \in [n]})^+. 
$$     
For $1\le i \le n$, $$(\tilde K_{\theta,m,p})_{ii} =d_i^{-1} \sum_{\sigma \in S_{m,p},\atop \sigma(l) =i} \mu_{\theta,m,p}(\sigma) = \left\{
     \begin{array}{lr}
       d_i^{-1} \frac{p}{\theta+m-1}  , ~\mathrm{if}~ 1\le i \le p,\\
       \\
       d_i^{-1} \frac{\theta+p-1}{\theta+m-1}  , ~\mathrm{if}~ p+1\le i \le n.
     \end{array}
     \right. $$
For $n+1 \le i \le m$,  $(\tilde K_{\theta,m,p})_{ii}=0$.    
\end{proof}

Obtaining a close form expression for Equation \eqref{eq:ewens2} in the general case seems to be much more challenging. However, we are able to obtain an inductive formula with the help of a result of Kurmayya and Sivakumar's result \cite{KS}.

\begin{theorem}[Theorem 3.2, \cite{KS}] \label{inverse}
Let $M=[A~~ a] \in \mathbb{R}^{m \times n}$ be a block matrix, with $A \in \mathbb{C}^{m\times (n-1)}$ and $a\in \mathbb{C}^{m}$ being written as a column vector. Let $B = M^* M$ and $s= \| a\|^2 -a^* A A^+ a$. Then if $s \neq 0$ 
$$
B^+ = \left( \begin{matrix} (A A^*)^+ + s^{-1} (A^+ a) (A^+ a)^* & -s^{-1} (A^+a) \\ -s^{-1}(A^+ a)^* & s^{-1} \end{matrix} \right),
$$ 
and if $s = 0$ ,
$$
B^+ = \left( \begin{matrix} (A A^*)^+ + \|b \|^2 (A^+ a) (A^+ a)^* - (A^+ a) (A^+ b)^* -(A^+ b) (A^+ a)^*& -\|b\|^2 A^+ a +A^+ b \\ -\|b\|^2(A^+a)^*+(A^+ b)^* & \|b\|^2 \end{matrix} \right),$$ 
where 
$$
b=(A^*)^+ (I+ A^+a (A^+ a)^*)^{-1} A^+ a.
$$
\end{theorem}

For a non--negative definite matrix $K$, one can decompose 
$$
K=UDU^*=\left( \begin{matrix} u_1\\ u_2\\ \vdots \\ u_m \end{matrix}\right) \left( \begin{matrix} d_1 & & & & \\ & d_1 & & & \\ & & \ddots & &\\ & & & &d_m \end{matrix} \right) \left( \begin{matrix} u_1^*&u_2^*& \ldots&u_m^* \end{matrix}\right),
$$
where $U$ is a unitary matrix. Then
\begin{equation*}
\begin{split}
W_\sigma &=V_\sigma K V_\sigma^T = \left( \begin{matrix} u_{\sigma(1)}\\ u_{\sigma(2)}\\ \vdots \\ u_{\sigma(p)} \end{matrix}\right) \left( \begin{matrix} d_1 & & & & \\ & d_1 & & & \\ & & \ddots & &\\ & & & &d_m \end{matrix} \right) \left( \begin{matrix} u_{\sigma(1)}^*&u_{\sigma(2)}^*& \ldots&u_{\sigma(p)}^* \end{matrix}\right) \\
&= \left( \begin{matrix} \tilde u_{\sigma(1)}\\ \tilde u_{\sigma(2)}\\ \vdots \\ \tilde u_{\sigma(p)} \end{matrix}\right) \left( \begin{matrix} \tilde u_{\sigma(1)}^*&\tilde u_{\sigma(2)}^*& \ldots&\tilde u_{\sigma(p)}^* \end{matrix}\right) := M^* M,
\end{split}
\end{equation*}
where 
$$
\tilde u_i = (\sqrt{d_1}u_i^i, \ldots, \sqrt{d_m}u_i^m ).
$$ 
Let $M=[M_1 ~ a]$ with $M_1 =\left( \begin{matrix} \tilde u_{\sigma(1)}^*&\tilde u_{\sigma(2)}^*& \ldots&\tilde u_{\sigma(p-1)}^* \end{matrix}\right) $ and $a=\tilde u_{\sigma(p)}^*$. Let $s= \| a\|^2 -a^* M_1 M_1^+ a$ and $b=(M_1^*)^+ (I+ M_1^+a (M_1^+ a)^*)^{-1} M_1^+ a$. By Theorem \ref{inverse}, 
$$
(M^* M)^+ = \left( \begin{matrix} (M_1 M_1^*)^+ & 0 \\ 0 & 0 \end{matrix} \right) + E_\sigma
$$ 
where the matrix $E_\sigma =$
\begin{equation}
 \left\{
     \begin{array}{lr}
       \left( \begin{matrix} s^{-1} (M_1^+ a) (M_1^+ a)^* & -s^{-1} (M_1^+a) \\ -s^{-1}(M_1^+ a)^* & s^{-1} \end{matrix} \right) \,\, ~\mathrm{if}~ s\neq 0,\\
       \\
       \left( \begin{matrix}   \|b \|^2 (M_1^+ a) (M_1^+ a)^* - (M_1^+ a) (M_1^+ b)^* -(M_1^+ b) (M_1^+ a)^*& -\|b\|^2 M_1^+ a +M_1^+ b \\ -\|b\|^2(A^+a)^*+(A^+ b)^* & \|b\|^2 \end{matrix} \right) \,\, ~\mathrm{if}~ s=0.
     \end{array}
     \right.
\end{equation}

Therefore,
\begin{equation}
\tilde K_{\theta,m,p} = \mathbb{E} (V_\sigma^T \left( \begin{matrix} (M_1 M_1^*)^+ & 0 \\ 0 & 0 \end{matrix} \right) V_\sigma) + \mathbb{E} (V_\sigma^T E_\sigma V_\sigma) = \tilde K_{\theta,m,p-1} + \mathbb{E} (V_\sigma^T E_\sigma V_\sigma).
\end{equation}

\section{Performance and Simulations}\label{sec:simulation}

In this section, we study the performance of our estimators and we compare them with other traditional methods. We focus on two types of true covariance matrix $\Sigma$ of size $m\times m$. In the first example,  $\Sigma=A_{\alpha}$ is an $m\times m$ Toeplitz covariance matrix with entries $\Sigma_{ij}=\alpha^{|i-j|}$. Here $0<\alpha<1$. Note that $\det(A_{\alpha})=(1-\alpha^2)^{m-1}$ and thus $A_{\alpha}$ is positive semi-definite if and only if $|\alpha|\le 1$. We call $A_{\alpha}$ the \emph{power Toeplitz matrix}. We observe that $A_\alpha$ is sparse in the sense that its entries decay in an exponential rate as they move away from the diagonal. In our experiment, we take $\alpha=0.5$. 

In the other example, we take $\Sigma=B_H$ to be the \emph{long-range dependence matrix} of the form
\begin{align*}
\Sigma_{ij} = \frac{1}{2} [ (|i-j| +1)^{2H} - 2|i-j|^{2H} + (|i-j|-1)^{2H} )]
\end{align*}
with $H\in [0.5,1]$. This kind of covariance matrix presents a process exhibiting long-range dependence, for example, the
increment process of fractional Brownian motion (see \cite{BL08a} for instance).  Contrary to the power Toepltiz matrix $A_\alpha$, the off-diagonal entries of $B_H$ (even far away from the diagonal) show long-range dependence and have non-negligible effort to the whole matrix. We choose $H=0.9$ in the simulation.


\subsection{Asymptotic behavior of the mean conjugate estimator under Ewens measure}\label{sec:asymptotic}

In this subsection, we study the asymptotic behavior for some covariance matrices using the mean conjugate estimator under Ewens measure. For an $m\times m$ symmetric matrix $K$, denote the eigenvalues $\lambda_1(K)\le \ldots \le \lambda_m(K)$. The simplest statistic of the eigenvalues is the \emph{empirical spectral measure}
$$\mu_{m}^K=\frac{1}{m} \sum_{j=1}^m \delta_{\lambda_j(K)}.$$
That is, for any set $E \subset \mathbb{R}$, $\mu_{m}(E)$ counts the proportion of eigenvalues of $K$ that lie in $E$. 

We show that if the diagonal entries of $K$ are all equal to 1 and the off-diagonal entries are not too big, then by choosing $\theta$ proportional to the dimension in the Ewens measure, $K_{\theta}=\E(M_{\sigma} K M_{\sigma}^*)$ is asymptotically equivalent to a convex combination of $K$ and the identity matrix $I$. 

For two positive functions $f(n), g(n)$, denote $f(n)=o(g(n))$  if $f(n)/g(n)\rightarrow 0$ as $n\rightarrow \infty$ and $f(n)=O(g(n))$ if $f(n)\le C g(n)$ for some $C>0$ for $n$ sufficiently large. 
\begin{theorem}\label{thm:asym}
For an $m\times m$ symmetric matrix $K=(a_{ij})$, assume $a_{ii}=1$ for all $1\le i \le m$,
\begin{align}\label{sparse}
\sum_{i\neq j} a_{ij}^2=O(m), \quad |\sum_{l\neq k} a_{lk}| = o(m^{3/2}) \quad \text{and} \quad \sum_{i\neq j}\big[ \sum_{k\neq i,j} (a_{ik}+a_{kj}) \big]^2 =  o(m^3).
\end{align}
Then for the mean conjugate estimator $K_\theta$ as in \eqref{eq:K} with $\theta = \beta m$, we have
$$\lim_{m\rightarrow \infty} \mu_m^{K_\theta}=\lim_{m\rightarrow \infty} \mu_m^{\frac{\beta^2}{(\beta+1)^2} K+(1-\frac{\beta^2}{(\beta+1)^2})I_m}.$$
\end{theorem}

\begin{proof} By Lemma 2.3 in \cite{bai99} the Levy metric of the empirical distributions of two $m\times m$ Hermitian matrix $A,B$ satisfies 
$$
L(\mu_m^A,\mu_m^B)\le \Big(\frac{1}{m}\mathrm{Tr}(A-B)(A-B)^*\Big)^{1/3}.
$$
It is known (see Theorem 6, Section 4.3, \cite{Galambos}) that the distribution functions $\mu_m $ converges weakly to $\mu$ if and only if the Levy metric $L(\mu_m,\mu)\rightarrow 0$.
Let $$E=K_\theta-\big( I_m+\frac{\beta^2}{(\beta+1)^2} (K-I_m) \big).$$

Thus it is enough to check that $\frac{1}{m} \mathrm{Tr} (E E^T) =\frac{1}{m}\sum_{i,j}E_{ij}^2\rightarrow 0$ as $m\rightarrow \infty$.

Note that $a_{ii}=1$ and $\theta=\beta m$. Applying Theorem \ref{permutation}, we obtain $E_{ii}=0$ and for $i\neq j$,
\begin{align*}
E_{ij}&=(K_\theta)_{ij} - \frac{\beta^2}{(\beta+1)^2} a_{ij}\\
&=\Big[\frac{\beta^2 m^2- \beta m -2}{(\beta m+m-2)(\beta m +m-1)} -\frac{\beta^2}{(\beta+1)^2}\Big] a_{ij} + \frac{\beta m-1}{(\beta m+m-2)(\beta m +m-1)} \sum_{k\neq i, j} (a_{ik}+a_{kj}) \\
&\quad\quad\quad+ \frac{1}{(\beta m+m-2)(\beta m +m-1)} \sum_{l\neq k} a_{lk} .
\end{align*}
Therefore, using the basic inequality $(a+b+c)^2 \le 3a^2+3b^2+3c^2$, we have
\begin{align*}
\frac{1}{m} \mathrm{Tr} (E E^T)&=\frac{1}{m}\sum_{i\neq j}E_{ij}^2\\
&\le \frac{3}{m}\left[\frac{\beta^2 m^2 - \beta m-2}{(\beta m+m-2)(\beta m +m-1)} -\frac{\beta^2}{(\beta+1)^2}\right]^2 \sum_{i\neq j} a_{ij}^2+ \frac{3}{m}\frac{\beta^2 m^2}{(\beta m+m-2)^4}\sum_{i\neq j}\big[ \sum_{k\neq i, j} (a_{ik}+a_{kj}) \big]^2\\
&\quad\quad + \frac{3}{m}\frac{m^2}{(\beta m+m-2)^4}\big(\sum_{l\neq k} a_{lk} \big)^2\\
&=o\Big(\frac{\sum_{i\neq j} a_{ij}^2}{m} \Big) + O\Big(\frac{1}{m^3} \sum_{i\neq j}\big[ \sum_{k\neq i,j} (a_{ik}+a_{kj}) \big]^2 \Big) + O\Big(\frac{1}{m^3} \big(\sum_{l\neq k} a_{lk} \big)^2 \Big)=o(1)
\end{align*}
by the assumption. This completes the proof.
\end{proof}

\begin{remark}
Theorem \ref{thm:asym} asserts if $K$ possesses some level of sparsity in terms of \eqref{sparse}, then asymptotically $K_{\theta}$ behaves like a linear convex combination of $I_m$ and the sample covariance matrix $K$. We only show the convergence of the overall behavior of the eigenvalues. Indeed, if we impose stronger conditions on the entries of $K$, i.e.  $$\sum_{i\neq j} a_{ij}^2=O(1), \quad |\sum_{l\neq k} a_{lk}| = o(m^{1/2}) \quad \text{and} \quad \sum_{i\neq j}\big[ \sum_{k\neq i,j} (a_{ik}+a_{kj}) \big]^2 =  o(m^2),$$
then the matrix $E$ in the proof of Theorem \ref{thm:asym} satisfies $\|E\|_F = o(1)$. By Weyl's inequality, one gets the individual eigenvalue of $K_\theta$ is close to that of $\frac{\beta^2}{(\beta+1)^2} K+(1-\frac{\beta^2}{(\beta+1)^2})I_m$. Similarly, by imposing extra conditions on the eigenvalues of $K$, one can obtain results on the perturbation of eigenvectors using the classical Davis-Kahan theorem (see for instance \cite[Section V]{SS90}). However, we found these imposed conditions are rather restrictive. It is an intriguing question to investigate the optimal conditions to guarantee the closeness of $K_\theta$ and $\frac{\beta^2}{(\beta+1)^2} K+(1-\frac{\beta^2}{(\beta+1)^2})I_m$.
\end{remark}

\begin{remark}\label{rem:LW}
In \cite{LW04}, Ledoit and Wolf introduce the \emph{linear shrinkage estimator} or the \emph{LW estimator} $$K_{LW}=\rho_1 I_m + \rho_2 K$$ to estimate the true covariance matrix $\Sigma$. They provide the optimal parameter $\rho_1^*$ and $\rho_2^*$ to minimize the error $\E \| K_{LW}-\Sigma\|_F$ in the space of $\{ \rho_1 I_m + \rho_2 K  :  \rho_1, \rho_1 \text{ non-random}\}$. The values of $\rho_1^*$ and $\rho_2^*$ actually depend on the true covariance matrix $\Sigma$. Specially, if $\Sigma_{ii}=1$ for all $i$, then $\rho_1^*+ \rho_2^*=1$ and $K_{LW}$ is the linear convex combination of $I_m$ and $K$. They suggest consistent estimators $\hat{\rho}_1$ and $\hat\rho_2$ (see Section 3.2 in \cite{LW04}) without prior knowledge of $\Sigma$. We will use the LW estimator $K_{LW}$ with parameters $\hat{\rho}_1$ and $\hat\rho_2$ for performance comparison. 
\end{remark}

\begin{remark} \label{rem:Linear}For the power Toeplitz matrix $A_{\alpha}=(\alpha^{|i-j|})_{1\le i,j\le m}$. Assume $0<\alpha<1$, it is easy to verify that $A_{\alpha}$ satisfies \eqref{sparse} and thus the conclusion of Theorem \ref{thm:asym} holds for $A_\alpha$.
Next let $K=(a_{ij})_{1\le i,j\le m}$ be the sample covariance matrix generated using Gaussian random variables. If the off-diagonal entries are not prominent (with high probability) in the sense of \eqref{sparse}, then the effect of the Ewens estimator with parameter $\theta=\beta m$ is asymptotically the same as the linear shrinkage estimator. Set $\beta=5$ and denote $\rho= \frac{\beta^2}{(\beta+1)^2}$. In Figure \ref{fig:compare}, we plot the difference $$\| K_\theta - \big(\rho I_m + (1-\rho) K \big)\|_{NF}$$ for $m=40, 80, 120, 160, 200$ and $n=m/2$, averaged over 50 repetitions. The blue line corresponds to the power Toeplitz matrix and the red dashed line is for the long-range dependence matrix. If the true covariance matrix $\Sigma$ is the power Toeplitz matrix, then the difference between $K_\theta$ and $\rho I_m + (1-\rho) K$ under the normalized Frobenius norm is getting smaller as $m,n$ getting larger.  However, if $\Sigma$ is the long-range dependence matrix, the difference between the Ewens estimator and the linear shrinkage estimator is getting bigger with the matrix size. This suggests the Ewens estimator has rather different behavior from the linear shrinkage estimator for the long-range dependence matrix.
\end{remark}
\begin{figure}[htbp]
\centering
\caption{Difference between the Ewens and linear shrinkage estimators for $\Sigma=A_\alpha$ (the blue diamonds) and  $\Sigma=B_H$ (the red triangles).}
\includegraphics[width=0.5\linewidth]{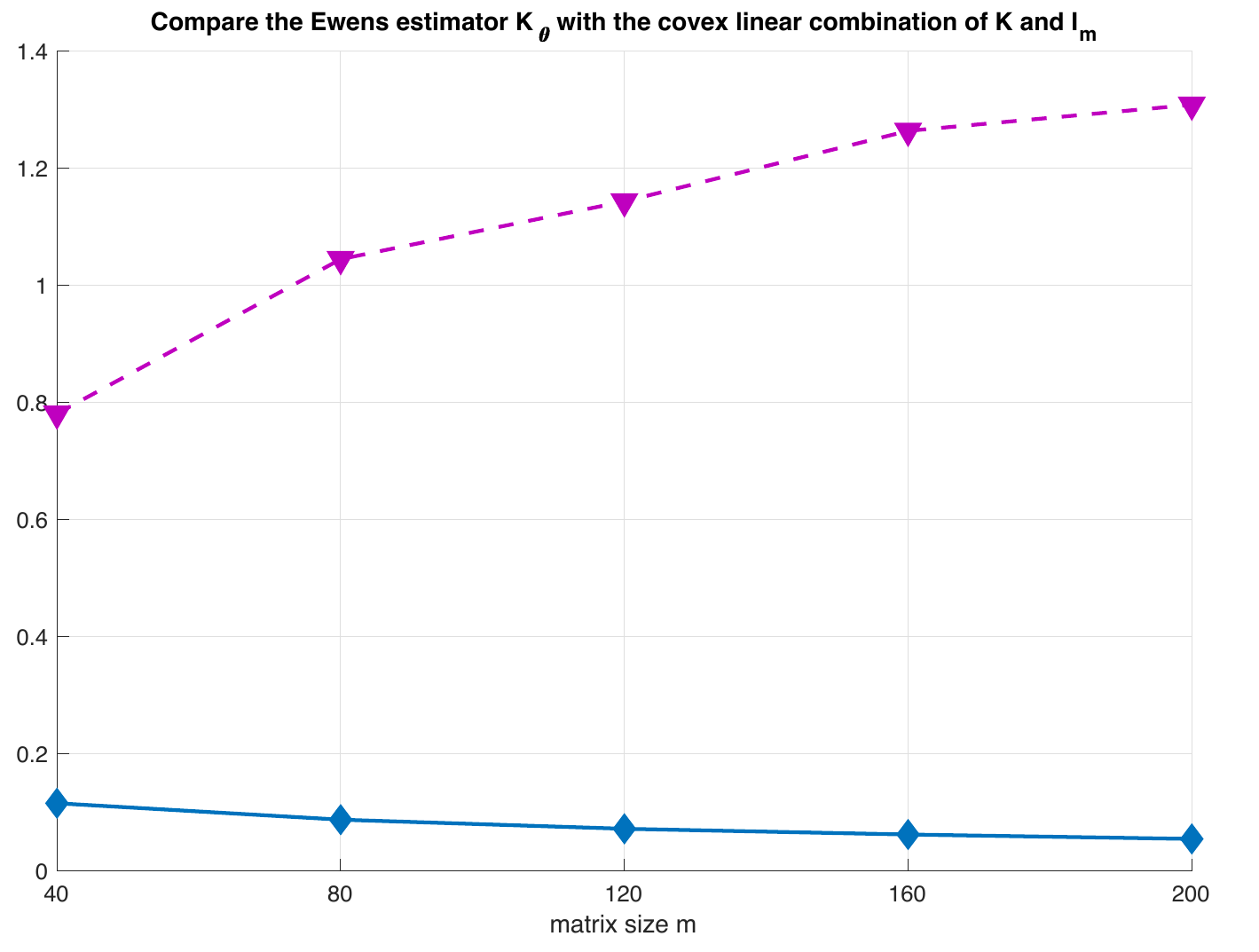}
\label{fig:compare}
\end{figure}

\subsection{Simulation study: finite sample}

\par In this subsection, we present some simulations to test the performance of our estimators. Let the random vector $X=(X^1,\ldots, X^m)^T$ have multivariate normal distribution  $N(0,\Sigma)$. Now we have $n$ measurements $(x_1, \ldots, x_n)$ where $x_i$'s are independent copies of $X$. Let $M=(x_1,\ldots,x_n)$ and form the sample covariance matrix $K=M M^T/n$. Assume $n<m$, we want to recover $\Sigma$ to the best of our knowledge. 

For brevity, we call the mean conjugate estimator under Ewens measure the  \emph{Ewens estimator}, and the linear shrinkage estimator by Ledoit and Wolf \cite{LW04} (see Remark \ref{rem:LW} above) the \emph{LW estimator}. We will compare the performance of the estimators $K_{LW}$, $\mathrm{invcov}_p(K)$ and $K_\theta=\mathbb{E}(M_\sigma K M^*_\sigma)$ as well as the sample covariance matrix $K$ itself. We will consider the error function $$\|K-\Sigma\|_{NF} = \big(\frac{1}{m} \sum_{i,j=1}^m (K_{ij} - \Sigma_{ij})^2 \big)^{1/2}$$ in terms of the \emph{normalized Frobenius norm} for an estimator $K$  of  $\Sigma$ for performance comparison. 

\textbf{\textit{Choosing the parameter $\theta$ for Ewens estimator.}} We first suggest how to choose the parameter $\theta$ for the Ewens estimator $K_\theta$. Given the sample covariance matrix $K$, the explicit formula of $K_\theta$ is provided in Theorem \ref{permutation}. We compute the formula of $\E \| K_\theta - \Sigma\|^2_{NF}$ in \eqref{eq:explicit} in Appendix \ref{appendix:formula}, which is denoted by $\mathcal{G}_\Sigma (\theta)$ for brevity. Note that  $\mathcal{G}_\Sigma (\theta)$ in \eqref{eq:explicit} is a rational function of the form $$\mathcal{G}_\Sigma (\theta)=\frac{a_4 \theta^4 + a_3 \theta^3 + a_2 \theta^2 + a_1 \theta + a_0}{(\theta+m-1)^2 (\theta+m-2)^2},$$ where the coefficients $a_i$'s depend on $m,n$ and the matrix $\Sigma$. An intuitive way to choose $\theta$ is to set 
$$\theta_0 = \mathrm{argmin}_{\theta>0} \mathcal{G}_\Sigma(\theta),$$
which is the best choice under the expected quadratic normalized Frobenius loss function. We call this $\theta_0$ the \emph{oracle} parameter.  If one has access to $\Sigma$ (or a few quantities of $\Sigma$ appearing in the formula \eqref{eq:explicit}), then $\theta_0$ is obtained by minimizing a rational function given $m$ and $n$, 
and we simply take $\theta = \theta_0$ in the Ewens estimator. However, in application, it is rare that any information of $\Sigma$ is known beforehand and only the sample covariance matrix $K$ is available. To choose $\theta$, we suggest the following method. 

Since the coefficients $a_i$'s in $\mathcal{G}_\Sigma (\theta)$ depend smoothly on $m,n$ and the matrix $\Sigma$, a small perturbation of $a_i$'s only leads to a small perturbation of the minimum value of $\mathcal{G}_\Sigma (\theta)$. Given the sample covariance matrix $K$, we replace $\Sigma$ in the expression of $\mathcal{G}_\Sigma (\theta)$ with $K$ and choose the parameter
\begin{align}\label{eq:hattheta}
\hat\theta= \mathrm{argmin}_{\theta>0} \mathcal{G}_K(\theta).
\end{align}
 We estimate the true covariance matrix $\Sigma$ using the Ewens estimator $K_{\hat\theta}$.

In Figure \ref{fig:theta1}, we plot the graphs of $\sqrt{\mathcal{G}_\Sigma (\theta)}$ as a function of $\theta>0$ for given pairs of $m,n$, for the power Toeplitz matrix and long-range dependence matrix respectively. In all plots, we can see that $\sqrt{\mathcal{G}_\Sigma (\theta)}$  achieves the unique minimum at an oracle value $\theta_0>0$. 

\begin{figure}[htbp]
\centering
\caption{Plots of $\sqrt{\mathcal{G}_\Sigma (\theta)}$ for $\Sigma=A_\alpha$ and  $\Sigma=B_H$.}
\includegraphics[width=0.8\linewidth]{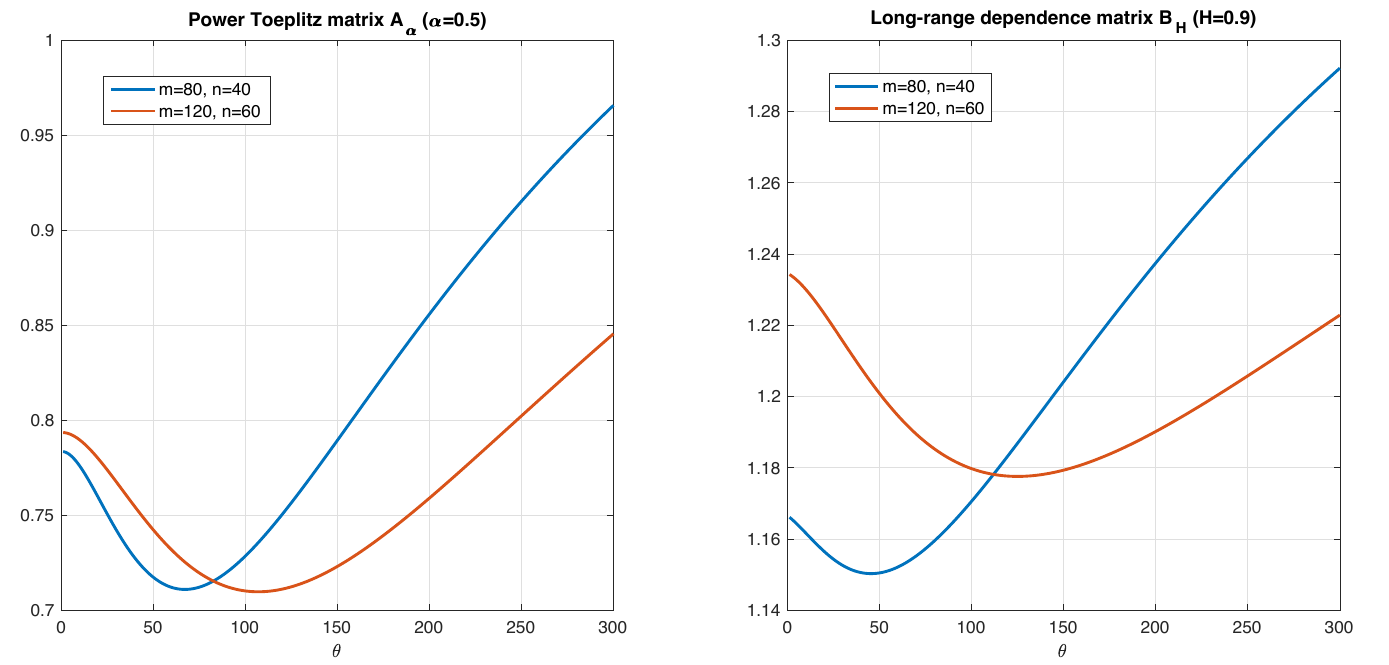}
\label{fig:theta1}
\end{figure}

In Table \ref{tab:parameter-power} and Table \ref{tab:parameter-long}, we numerically compute the oracle parameter $\theta_0$ and its corresponding loss value $( \E \| K_{\theta_0} - \Sigma \|_{NF}^2)^{1/2} = \sqrt {\mathcal{G}_\Sigma (\theta_0)}$. We also find the estimated $\hat\theta$ and its loss value $\|K_{\hat\theta}-\Sigma\|_{NF}$, as well as the loss value $\| K - \Sigma \|_{NF}$ of using the sample covariance matrix $K$ directly. These three quantities are averaged over 50 repetitions. In both tables, we note that both $\theta_0$ and $\hat\theta$ increase with the matrix size $m$ and decrease with the ratio $n/m$. However, our suggested $\hat\theta$ is quite far from the oracle $\theta_0$. This happens possibly because the coefficients $a_i$'s are perturbed by a large value when we replace $\Sigma$ with $K$. It is not clear to us yet how to select a better parameter $\theta$. Comparing Table \ref{tab:parameter-power} with Table \ref{tab:parameter-long}, we see that for the long-range dependence matrix, $\|K_{\hat\theta}-\Sigma\|_{NF}$ differs very little from $\sqrt{\mathcal{G}_\Sigma(\theta_0)}$, even though $\hat\theta$ is not a good approximation of $\theta_0$. In all cases, directly using the sample covariance matrix $K$ provides the worst performance. 

\begin{table}[htbp]
\caption{Power toepltiz matrix $\Sigma=A_\alpha \ (\alpha=0.5)$:  oracle and estimated $\theta$ and their corresponding loss values and loss of the sample covariance matrix.}
\label{tab:parameter-power}
\centering
\begin{tabular}{ l || c | c | c | c| c}
$n=m/2$& $\theta_0$ & $\sqrt{\mathcal{G}_\Sigma(\theta_0)}$ & $\hat\theta$ & $\|K_{\hat\theta}-\Sigma\|_{NF}$ & $\| K - \Sigma \|_{NF}$ \\ \hline
$m=40,n=20$ & 27.47 & 0.7145 & 106.01 & 0.8929 & 1.4344 \\
$m=80,n=40$ & 67.11 & 0.7109 & 226.27 & 0.8908 & 1.4296\\
$m=120,n=60$ & 106.99 & 0.7097 &  350.02 & 0.8857 & 1.4240\\
$m=160,n=80$ & 146.93 & 0.7091 & 472.59 &0.8836  &  1.4206 \\ \hline\hline
$n=m/4$& $\theta_0$ & $\sqrt{\mathcal{G}_\Sigma(\theta_0)}$ & $\hat\theta$ & $\|K_{\hat\theta}-\Sigma\|_{NF}$ & $\| K - \Sigma \|_{NF}$ \\ \hline
$m=40,n=10$ & 12.36 & 0.7661 & 88.95 & 1.1517 & 2.0448\\
$m=80,n=20$ & 36.52 & 0.7602 & 199.56 & 1.1473 & 2.0235\\
$m=120,n=30$ & 60.78 & 0.7586 & 308.10 & 1.1409 & 2.0081\\
$m=160,n=40$ &  85.06 & 0.7579 & 418.75 & 1.1416 & 2.0098 \\ \hline
\end{tabular}
\end{table}

\begin{table}[htbp]
\caption{Long-range dependence matrix $\Sigma=B_H\ (H=0.9)$: oracle and estimated $\theta$ and their corresponding loss values and loss of the sample covariance matrix.}
\label{tab:parameter-long}
\centering
\begin{tabular}{ l || c | c | c | c| c}
$n=m/2$& $\theta_0$ & $\sqrt{\mathcal{G}_\Sigma(\theta_0)}$ & $\hat\theta$ & $\|K_{\hat\theta}-\Sigma\|_{NF}$ & $\| K - \Sigma \|_{NF}$ \\ \hline
$m=40,n=20$ & 4.30 & 1.1263 & 73.83 & 1.1696 & 1.6254 \\
$m=80,n=40$ & 45.30 & 1.1503 & 195.66 & 1.1829 & 1.5107\\
$m=120,n=60$ & 124.86 & 1.1776 & 325.09 & 1.1814 & 1.5194\\
$m=160,n=80$ & 228.00 &1.1978 &512.75 & 1.2074 & 1.4825\\ \hline\hline
$n=m/4$& $\theta_0$ & $\sqrt{\mathcal{G}_\Sigma(\theta_0)}$ & $\hat\theta$ & $\|K_{\hat\theta}-\Sigma\|_{NF}$ & $\| K - \Sigma \|_{NF}$ \\ \hline
$m=40,n=10$ & 1.88 & 1.4787 & 80.60 & 1.5858 & 2.1031\\
$m=80,n=20$ & 5.51 & 1.4322 & 152.59 & 1.5186 & 2.1461\\
$m=120,n=30$ & 23.49 & 1.4504 & 261.23 & 1.5407 & 2.0868\\
$m=160,n=40$ & 69.52 & 1.4782 &  367.96 & 1.5396 & 2.0972  \\ \hline
\end{tabular}
\end{table}

\textbf{\textit{Performance comparision.}} We compare the performance of the Ewens estimator, LW estimator, the $\mathrm{Invcov}_p$ estimator and the sample covariance matrix, for both models: power Toeplitz matrix $A_\alpha \ (\alpha=0.5)$ and long-range dependence matrix $B_H \ (H=0.8)$.

For the $\mathrm{Invcov}_p$ estimator, we approximate the true covariance matrix $\Sigma$ by $(p/m)\mathrm{invcov}_p(K)^{-1}$ and consider the loss function $$\|(p/m)\mathrm{invcov}_p(K)^{-1}-\Sigma\|_{NF}.$$ 
Due to the complicated expression of the $\mathrm{Invcov}_p$ operator, it is hard to suggest how to turn the parameter $p$. In Figure \ref{fig:parap}, we plot the graphs of $\|(p/m)\mathrm{invcov}_p(K)^{-1}-\Sigma\|_{NF}$ for all values of $5\le p \le n$ for given pairs of $m,n$. For the power Toeplitz matrix, the optimum values of $p$ are approximately $p=8$ for $m=40,n=20$, $p=13$ for $m=80, n=40$, $p=18$ for $m=120,n=60$ and $p=26$ for $m=160, n=80$. For the long-range dependence matrix, the optimum values of $p$ happen at its largest possible value $n$. We take these optimum values $p$ in later comparison. Although it does not seem a fair game for other estimators, we will see that the $\mathrm{Invcov}_p$ estimator is never the best estimator, even with the optimum parameter $p$. 

\begin{figure}[htbp]
\centering
\caption{Plots of $\|(p/m)\mathrm{invcov}_p(K)^{-1}-\Sigma\|_{NF}$ for  $\Sigma=A_\alpha$ and  $\Sigma=B_H$.}
\includegraphics[width=0.8\linewidth]{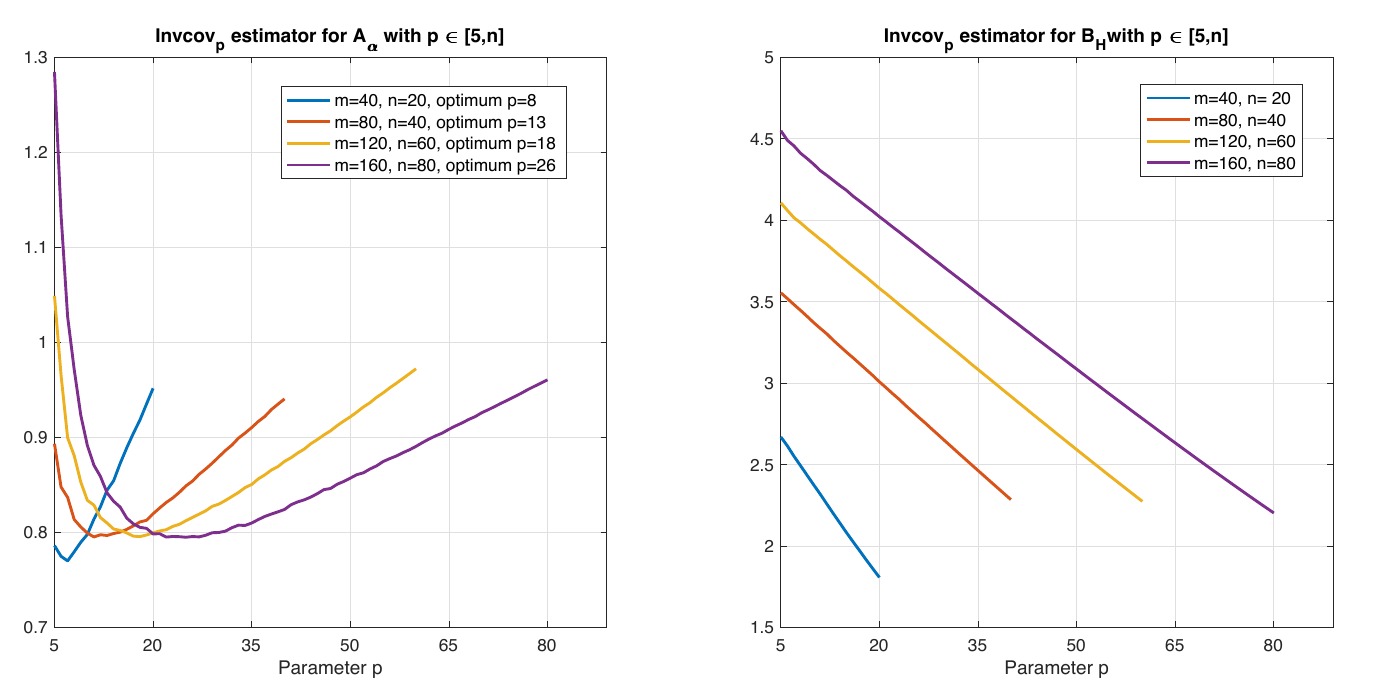}
\label{fig:parap}
\end{figure}

In Figure \ref{fig:allestimator}, we compare the performance of the estimators. We plot the loss function values $$\|\mathrm{Estimator} - \Sigma\|_{NF}$$ for $m=40, 80, 120, 160$ and $n=m/2$, averaged over 50 repetitions, for $\Sigma$ the power Toeplitz matrix and the long-range dependence matrix.

\begin{figure}[htbp]
\centering
\caption{Compare different estimators for  $\Sigma=A_\alpha$ and  $\Sigma=B_H$.}
\includegraphics[width=0.8\linewidth]{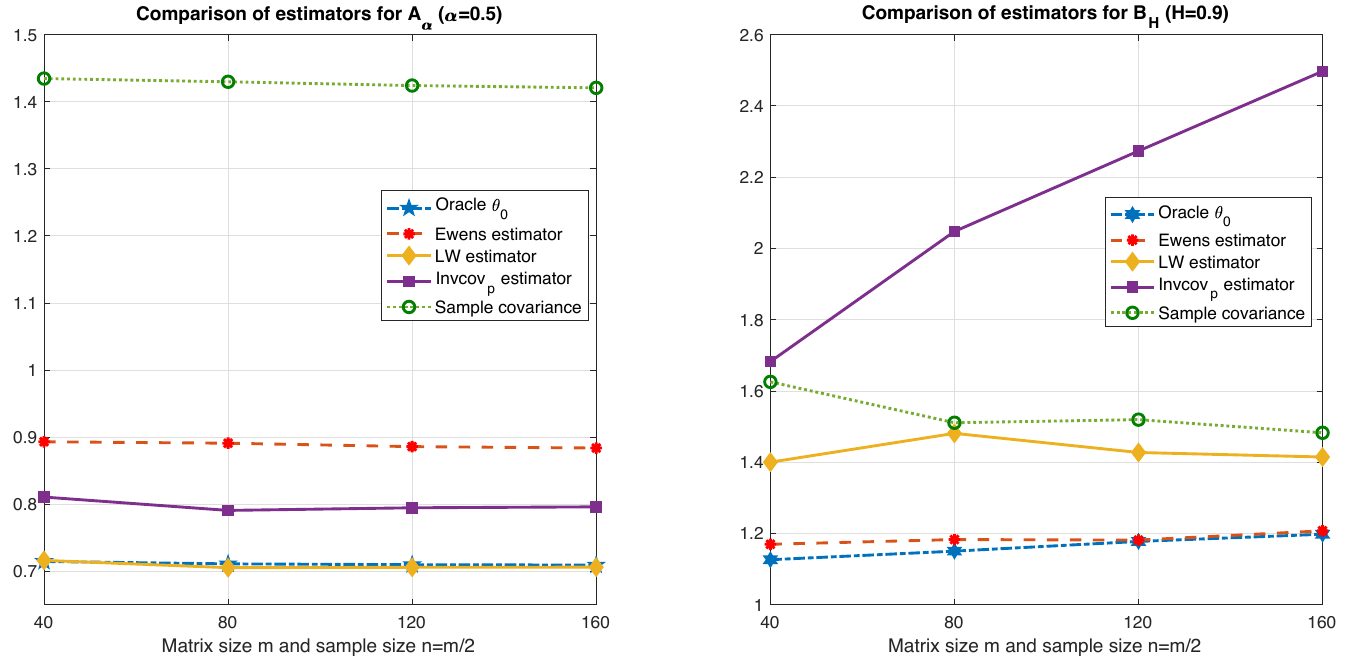}
\label{fig:allestimator}
\end{figure}

For the power Toeplitz matrix (left figure in Figure \ref{fig:allestimator}), we observe that the LW estimator (yellow line) has the best performance and for the oracle $\theta_0$ (red dashed line), the Ewens estimator has almost the identical performance. This is in accordance with Theorem \ref{thm:asym} (see also Remark \ref{rem:Linear}), that is, the Ewens estimator is asymptotically equivalent to the linear shrinkage estimator $\rho I_m + (1-\rho) K$. In our finite sample study, we further observe that the Ewens estimator with oracle $\theta_0$ performs roughly the same as the linear shrinkage estimator with the best $\rho$ which is provided in the LW estimator. However, our suggested parameter $\hat\theta$ does not seem a good approximation. The $\mathrm{invcov}_p$ (purple line) with optimum $p$ outperforms the Ewens estimator with $\hat\theta$, but is not comparable with the LW estimator. Directly using the sample covariance matrix $K$ (green dotted line) provides the worst approximation. Nevertheless, when $\Sigma$ is the power Toeplitz matrix and possesses some level of sparsity, the LW estimator is the best choice. By providing a better parameter $\hat\theta$, the Ewens estimator might be comparable with the LW estimator.

For the long-range dependence matrix (right figure in Figure \ref{fig:allestimator}), we see that the Ewens estimator (for both oracle $\theta_0$ and estimated $\hat\theta$) outperforms the other estimators. Actually, the Ewens estimator $K_{\hat\theta}$ performs almost as good as the oracle $K_{\theta_0}$. The LW estimator is only slightly better than using the sample covariance matrix directly. The $\mathrm{invcov}_p$ estimator (even with optimum $p$) always gives the largest errors and is not a good estimator for the long-range dependence matrix. 



\subsection{Comments} The simulations suggest that for the true covariance matrix with power decay Toeplitz structure, the Ewens estimator with the oracle parameter is asymptotically as good as the LW estimator.  At present, we do not have a satisfying algorithm for choosing the parameter $\theta$ very close to the oracle value. For the current suggested parameter $\hat\theta$, the LW estimator outperforms the Ewens estimator. However, for the true covariance matrix that has long-range dependence structure, the Ewens estimator always performs better than all other estimators considered. Even our suggested parameter $\hat\theta$ is not an accurate approximation to the oracle parameter, it has little influence on the performance. Provided a more accurate algorithm for choosing the parameter $\theta$, the Ewens estimator seems a better choice than the LW estimator since it is less sensitive to the sparsity of the true covariance matrix. There are still many questions to be answered: How does the operator $K_\theta$ change the eigenvalues and eigenvectors of the original matrix $K$? Is there a better way to select the parameter for the Ewens estimator, using the samples? Is it possible to analyze the performance of the Ewens estimator under other loss functions?  A more comprehensive understanding on  the Ewens estimator $K_\theta$ will shed lights on analyzing the performance of the hybrid operators $K_{\theta,m,p}$ and $\tilde{K}_{\theta,m,p}$ defined in Section \ref{sec:mix}. We did not include simulations on the performance of these hybrid operators in this paper. However, it is an intriguing future research question to explore how the parameters $p$ and $\theta$ affect the estimations. 

\appendix
\section{Small dimensional examples for computing $\mathbb{E}(\Phi^* (\Phi D_n \Phi^*)^l \Phi)$}\label{sec:small}
In this appendix, we provide small dimensional examples for computing $\mathbb{E}(\Phi^* (\Phi D_n \Phi^*)^l \Phi)$ using formulas derived in Section \ref{sec:formula}.

Let $\lambda_j=(N-j,1^j)$ be the partition of $N$ with $j$ ones. This one has a hook shape with $N-j$ blocks in the row and $j+1$ blocks in the column.
$$\begin{Young} 
 &  & & & & & & & & \cr
 \cr
 \cr
 \cr
\end{Young} $$

For $l=1$, it was shown in \cite{MTS} that 
$$
\mathbb{E}(\Phi^* (\Phi D_n \Phi^*)^l \Phi)=\frac{p(np-1)}{n(n^2-1)} D_n+ \frac{p(n-p)}{n(n^2-1)} \mathrm{Tr}(D_n) I_n.
$$
For $l=2$ and $\rho=(1,1,1), (1,2), (3) \vdash 3$, we list all border--strip tableaux of shape $\lambda_j$ and type $\rho$ in the following table. 

\begin{center}
    \begin{tabular}{ |l| l | l | l| l|}
    \hline
     & $\rho=(1,1,1)$ & $\rho=(1,2)$ & $\rho=(3)$ \\ \hline
    $\lambda_0=(3)$  & $\begin{Young} 1& 2& 3\cr \end{Young}$ &  $\begin{Young} 1& 2& 2\cr \end{Young}$ & $\begin{Young} 1& 1& 1\cr \end{Young}$\\ \hline
    $\lambda_1=(2,1)$  & $\begin{Young} 1& 2 \cr 3\cr \end{Young}$ \,\,\&\,\,  $\begin{Young} 1& 3 \cr 2\cr \end{Young}$&  Does not exist & $\begin{Young} 1& 1\cr 1\cr \end{Young}$\\ \hline
    $\lambda_2=(1,1,1)$  & $\begin{Young} 1\cr 2 \cr 3\cr \end{Young}$ & $\begin{Young} 1\cr 2\cr 2\cr \end{Young}$  & $\begin{Young} 1\cr 1\cr 1\cr \end{Young}$\\ \hline
    \end{tabular}
\end{center}
 
Thus,

\begin{center}
    \begin{tabular}{ |l| l| l| l| }  
    \hline
    $\chi^{\lambda_j}(\rho)$ & $\rho=(1,1,1)$ & $\rho=(1,2)$ & $\rho=(3)$ \\ \hline
    $\lambda_0=(3)$  & 1 & 1 & 1\\ \hline
    $\lambda_1=(2,1)$  & 2 &  0 & -1\\ \hline
    $\lambda_2=(1,1,1)$  & 1 & -1  & 1\\ \hline
    \end{tabular}
\end{center}    

\begin{equation*}
\begin{split}
&s_{\lambda_0}(D) = \frac{\mathrm{Tr}(D)^3}{3!}+\frac{\mathrm{Tr}(D) \mathrm{Tr}(D^2)}{2}+ \frac{\mathrm{Tr}(D^3)}{3}, \quad\quad ~\frac{\partial s_{\lambda_0}}{\partial d_i} = d_i^2+ \mathrm{Tr}(D) d_i + \frac{\mathrm{Tr}(D)^2+\mathrm{Tr}(D^2)}{2}\\
&s_{\lambda_1}(D) =2 \frac{\mathrm{Tr}(D)^3}{3!}- \frac{\mathrm{Tr}(D^3)}{3},\quad\quad ~\frac{\partial s_{\lambda_1}}{\partial d_i} =- d_i^2+ \mathrm{Tr}(D)^2\\
&s_{\lambda_2}(D) = \frac{\mathrm{Tr}(D)^3}{3!}-\frac{\mathrm{Tr}(D) \mathrm{Tr}(D^2)}{2}+ \frac{\mathrm{Tr}(D^3)}{3},\quad\quad ~\frac{\partial s_{\lambda_2}}{\partial d_i} = d_i^2- \mathrm{Tr}(D) d_i +\frac{\mathrm{Tr}(D)^2-\mathrm{Tr}(D^2)}{2}.
\end{split}
\end{equation*}

Furthermore,
\begin{equation*}
\begin{split}
\left( \mathbb{E}(\Phi^* (\Phi D \Phi^*)^2 \Phi)\right)_{ii}&=\frac{1}{3}\sum_{j=0}^2 (-1)^j \frac{(2+p-j)!(n-j-1)!}{(2+n-j)!(p-j-1)!} \frac{\partial s_{\lambda_j}(D)}{\partial d_i}\\
&= (c_0+c_1+ c_2) d_i^2 + (c_0-c_2) \mathrm{Tr}(D) d_i+c_0\frac{\mathrm{Tr}(D)^2+\mathrm{Tr}(D^2)}{2}-c_1\\
&+c_2\frac{\mathrm{Tr}(D)^2-\mathrm{Tr}(D^2)}{2},
\end{split}
\end{equation*}
where 
\begin{equation*}
c_0=\frac{1}{3} \frac{(2+p)! (n-1)!}{(2+n)! (p-1)!},\hspace{0.3cm}
c_1=\frac{1}{3} \frac{(1+p)! (n-2)!}{(1+n)! (p-2)!},\hspace{0.3cm}
c_2=\frac{1}{3} \frac{p! (n-3)!}{n! (p-3)!}.
\end{equation*}

Finally, 
\begin{equation*}
\begin{split}
\mathbb{E}\Big(\Phi^* (\Phi D \Phi^*)^2 \Phi \Big)&=(c_0+c_1+ c_2) D^2 + (c_0-c_2) \mathrm{Tr}(D) D\\
&+\Big(c_0\frac{\mathrm{Tr}(D)^2+\mathrm{Tr}(D^2)}{2}-c_1 \mathrm{Tr} (D)^2+c_2\frac{\mathrm{Tr}(D)^2-\mathrm{Tr}(D^2)}{2} \Big) I_n.
\end{split}
\end{equation*}

\section{Computing $\E \| K_\theta - \Sigma \|_{NF}^2$}\label{appendix:formula}
In this section, we compute the explicit formula for $\E \| K_\theta - \Sigma \|_{NF}^2 = \frac{1}{m}\E \| K_\theta - \Sigma \|_{F}^2$ and express the formula in terms of $\Sigma$. We assume the $m$-dimensional random vector $X$ has the normal distribution $N(0, \Sigma)$. Let $x_1,\ldots, x_n$ be $n$ independent copies of $X$. Recall $M=(x_1,\ldots,x_n)$ and $K=M M^T/n=(a_{ij})$.  Then $$\E \| K_\theta - \Sigma \|_F^2=\sum_{i=1}^m \E (K_\theta- \Sigma)_{ii}^2 + \sum_{i\neq j } \E(K_\theta- \Sigma)_{ij}^2.$$

By Theorem \ref{permutation}, we first have
\begin{align*}
(K_\theta- \Sigma)_{ii}^2 &= \Big( \frac{\theta-1}{\theta+m-1} a_{ii} + \frac{1}{\theta+m-1} \mathrm{Tr} K - \Sigma_{ii} \Big)^2\\
&=  \frac{(\theta-1)^2}{ (\theta+m-1)^2 } a_{ii}^2 + \frac{1}{(\theta+m-1)^2} (\mathrm{Tr} K)^2 + \Sigma_{ii}^2 + \frac{2(\theta-1)}{(\theta+m-1)^2} a_{ii} \tr K\\
&\quad - \frac{2(\theta-1)}{\theta+m-1} a_{ii}\Sigma_{ii} - \frac{2}{\theta+m-1} \Sigma_{ii}\tr K.
\end{align*}
Note that $\Sigma_{ii} = \E a_{ii}$ and $\E \tr K = \sum_{i=1}^m \Sigma_{ii}$. Thus 
\begin{align*}
\sum_{i=1}^m \E (K_\theta- \Sigma)_{ii}^2 &=  \frac{(\theta-1)^2}{ (\theta+m-1)^2 } \big( \sum_{i=1}^m \E a_{ii}^2 \big) + \frac{m \E (\mathrm{Tr} K)^2}{(\theta+m-1)^2}  + \sum_{i=1}^m \Sigma_{ii}^2 + \frac{2(\theta-1)}{(\theta+m-1)^2} \E(\tr K)^2\\
&\quad - \frac{2(\theta-1)}{\theta+m-1} \sum_{i=1}^m \Sigma_{ii}^2 - \frac{2}{\theta+m-1} (\E \tr K)^2.
\end{align*}
Plugging in$$\E (\tr K )^2 = \sum_{i=1}^m \E a_{ii}^2 + \sum_{i \neq j} \E a_{ii} a_{jj},$$ we get
\begin{align*}
\sum_{i=1}^m \E (K_\theta- \Sigma)_{ii}^2 &=  \frac{\theta^2+m-1}{ (\theta+m-1)^2 } \big( \sum_{i=1}^m \E a_{ii}^2 \big) - \frac{\theta - m -1}{\theta+m-1} \sum_{i=1}^m \Sigma_{ii}^2 + \frac{2\theta+m-2}{(\theta+m-1)^2} \sum_{i\neq j} \E a_{ii} a_{jj}\\
&\quad - \frac{2}{\theta+m-1} \big( \sum_{i=1}^m \Sigma_{ii} \big)^2.
\end{align*}
For brevity, denote $\beta= (\theta+m-1)(\theta+m-2)$. Next, by the formula obtained in Theorem \ref{permutation}, we get for $i\neq j$
\begin{align*}
(K_\theta- \Sigma)_{ij}^2 &= \frac{1}{\beta^2}\big( (\theta^2-1) a_{ij} + (\theta-1) a_{ji} + (\theta-1) \sum_{k \neq i,j} (a_{ik}+ a_{kj} ) + \sum_{l \neq k} a_{lk}  -\beta \Sigma_{ij} \big)^2\\
&=\frac{1}{\beta^2}\big( \theta(\theta-1) a_{ij} + (\theta-1) \sum_{k\neq i} a_{ik} + (\theta-1) \sum_{k \neq j} a_{jk} + \sum_{l\neq k} a_{lk} - \beta \Sigma_{ij} \big)^2.
\end{align*}
Expanding the square above and taking the expectation over the sum of all $i\neq j$, one obtains
\begin{align*}
\sum_{i\neq j}\E(K_\theta- \Sigma)_{ij}^2 &=  \frac{1}{\beta^2}\Big[ \theta^2(\theta-1)^2 \sum_{i\neq j} \E a_{ij}^2 + 2(\theta-1)^2 \sum_{i\neq j} \E(\sum_{k\neq i} a_{ik})^2 + m(m-1) \E (\sum_{i\neq j} a_{ij})^2 + \beta^2 \sum_{i\neq j} \Sigma_{ij}^2 \\
&\quad \quad\quad  +4\theta(\theta-1)^2 \sum_{i\neq j} \sum_{k\neq i} \E a_{ij} a_{ik} + 2\theta(\theta-1) \E (\sum_{i\neq j} a_{ij})^2 - 2\beta \theta(\theta-1) \sum_{i \neq j} \Sigma_{ij}^2\\
&\quad \quad\quad  + 2(\theta-1)^2 \sum_{i \neq j} \E (\sum_{k\neq i} a_{ik})(\sum_{k\neq j} a_{jk}) + 4(\theta-1) \E \big( \sum_{i\neq j} \sum_{k\neq i} a_{ik} \big) (\sum_{l\neq k} a_{lk}) \\
&\quad\quad\quad   - 4\beta(\theta-1) \sum_{i\neq j}  \sum_{k\neq i}\Sigma_{ij}\Sigma_{ik}- 2\beta(\sum_{l\neq k} \Sigma_{lk})^2\Big].
\end{align*}
We observe in the above summation that
\begin{align*}
&\sum_{i\neq j} \E(\sum_{k\neq i} a_{ik})^2 = (m-1)\sum_{i=1}^m \E (\sum_{k\neq i} a_{ik})^2,
\end{align*}
$$
\sum_{i\neq j}  \sum_{k\neq i}\E a_{ij} a_{ik} = \sum_{i=1}^m \E (\sum_{k\neq i} a_{ik})^2,
$$
\begin{align*}
\sum_{i \neq j} \E (\sum_{k\neq i} a_{ik})(\sum_{k\neq j} a_{jk})=\E (\sum_{l\neq k} a_{lk})^2 - \sum_{i=1}^m \E (\sum_{k\neq i} a_{ik})^2
\end{align*}
and
\begin{align*}
\E \big( \sum_{i\neq j} \sum_{k\neq i} a_{ik} \big) (\sum_{l\neq k} a_{lk}) = (m-1) \E \big( \sum_{i=1}^m \sum_{k\neq i} a_{ik} \big) (\sum_{l\neq k} a_{lk}) = (m-1) \E (\sum_{l\neq k} a_{lk})^2.
\end{align*}
Thus, after simplification, we get
\begin{align*}
\sum_{i\neq j}\E(K_\theta- \Sigma)_{ij}^2 &=  \frac{1}{\beta^2}\big[ \theta^2(\theta-1)^2 (\sum_{i\neq j} \E a_{ij}^2) + 2(\theta-1)^2(2\theta+m-2) \sum_{i=1}^m \E (\sum_{j\neq i} a_{ij})^2 \\
&\quad+\big[ m(m-1) + 2(\theta-1) (2\theta + 2m-3)  \big] \E(\sum_{i\neq j} a_{ij})^2  - 4\beta(\theta-1)\sum_{i=1}^m (\sum_{i\neq j} \Sigma_{ij})^2 \\
&\quad-2\beta(\sum_{i\neq j} \Sigma_{ij})^2 +\big(\beta^2-2\beta\theta(\theta-1) \big) \sum_{i\ne j} \Sigma_{ij}^2\big].
\end{align*}

Finally, we get the explicit formula 
\begin{align}\label{eq:explicit1}
\E \| K_\theta - \Sigma \|_F^2 &=  \frac{\theta^2+m-1}{ (\theta+m-1)^2 } \big( \sum_{i=1}^m \E a_{ii}^2 \big) - \frac{\theta - m -1}{\theta+m-1} \sum_{i=1}^m \Sigma_{ii}^2 + \frac{2\theta+m-2}{(\theta+m-1)^2} \sum_{i\neq j} \E a_{ii} a_{jj}- \frac{2}{\theta+m-1} \big( \sum_{i=1}^m \Sigma_{ii} \big)^2\nonumber\\
&\quad  + \frac{\theta^2(\theta-1)^2}{(\theta+m-1)^2(\theta+m-2)^2}(\sum_{i\neq j} \E a_{ij}^2) + \frac{2(\theta-1)^2(2\theta+m-2) }{(\theta+m-1)^2(\theta+m-2)^2 } \sum_{i=1}^m \E (\sum_{j\neq i} a_{ij})^2\nonumber\\
&\quad + \frac{2(\theta-1)(2\theta+2m-3)+m(m-1)}{ (\theta+m-1)^2(\theta+m-2)^2} \E(\sum_{i\neq j} a_{ij})^2 - \frac{4(\theta-1)}{(\theta+m-1)(\theta+m-2)} \sum_{i=1}^m(\sum_{j\neq i} \Sigma_{ij})^2 \nonumber\\
&\quad - \frac{2}{(\theta+m-1)(\theta+m-2)} (\sum_{i\neq j}\Sigma_{ij})^2 + \Big[1-\frac{2\theta(\theta-1)}{(\theta+m-1) (\theta+m-2)} \Big](\sum_{i\neq j} \Sigma_{ij}^2).
\end{align}
Since we assume $X=(X^1,\ldots, X^m)^T \sim N(0,\Sigma)$, we can further express \eqref{eq:explicit1} in terms of the entries of $\Sigma$. We use $x_s^i$ to denote the $i$th entry of the vector $x_s$. Note that $a_{ij} = \frac{1}{n}\sum_{s=1}^n x_s^i x_s^j$ by our definition of $K$. Besides, $\E K= \Sigma$. We also use the following facts about  multivariate normal distribution:
\begin{align*}
\E (X^i)^2 = \Sigma_{ii}, \quad \E (X^i)^4 = 3 \Sigma_{ii}^2, \quad \E X^i X^j = \Sigma_{ij}
\end{align*}
and
\begin{align*}
\E X^i X^{k_1} X^j X^{k_2} = \Sigma_{i k_1} \Sigma_{j k_2} + \Sigma_{i j} \Sigma_{k_1 k_2} +  \Sigma_{i k_2} \Sigma_{j k_1} \quad\text{for arbitrary } 1\le i, j, k_1,k_2\le m.
\end{align*}
It is elementary to verify the following calculation.
\begin{align}\label{eq:term1}
\sum_{i=1}^m \E a_{ii}^2 = \frac{1}{n^2}\sum_{i=1}^m \big[ \E \sum_{s=1}^n (x_s^i)^4 + \sum_{s\neq t} \E (x_s^i)^2 \E(x_t^i)^2\big] = \frac{1}{n^2}\sum_{i=1}^m \big[ 3n \Sigma_{ii}^2 + n(n-1) \Sigma_{ii}^2 \big] = \frac{n+2}{ n} \sum_{i=1}^m \Sigma_{ii}^2
\end{align}
and
\begin{align}\label{eq:term2}
\sum_{i\neq j} \E a_{ii} a_{jj} &= \frac{1}{n^2}\sum_{i \neq j} \sum_{s,t=1}^n \E (x_s^i)^2 (x_t^j)^2 = \frac{1}{n^2}\sum_{i\neq j} \big[ n \E (X^i)^2 (X^j)^2 + n(n-1) \Sigma_{ii} \Sigma_{jj}\big] \nonumber\\
&=\sum_{i \neq j} \Sigma_{ii} \Sigma_{jj} + \frac{2}{n} \sum_{i\neq j} \Sigma_{ij}^2
\end{align}
and
\begin{align}\label{eq:term3}
\sum_{i\neq j} \E a_{ij}^2 &=\frac{1}{n^2}\sum_{i\neq j} \big[ n \E (X^i)^2 (X^j)^2 + n(n-1) (\E X^i X^j)^2\big] = \frac{1}{n} \sum_{i\neq j} \Sigma_{ii} \Sigma_{jj} + \frac{n+1}{ n} \sum_{i \neq j} \Sigma_{ij}^2.
\end{align}
Similarly, we also obtain
\begin{align}\label{eq:term4}
\sum_{i=1}^m \E(\sum_{j\neq i} a_{ij})^2 &= \frac{1}{n^2} \sum_{i=1}^m \sum_{j_1, j_2 \neq i} (n \Sigma_{ii} \Sigma_{j_1 j_2} + 2n \Sigma_{i j_1} \Sigma_{i j_2} + n(n-1) \Sigma_{i j_1} \Sigma_{i j_2})  \nonumber\\
&= \frac{1}{n}\sum_{i=1}^m\sum_{j_1,j_2\neq i} \Sigma_{ii} \Sigma_{j_1 j_2} + \frac{n+1}{n} \sum_{i=1}^m (\sum_{j\neq i} \Sigma_{ij})^2
\end{align}
and
\begin{align}\label{eq:term5}
\E (\sum_{i\neq j} a_{ij})^2 &= \frac{1}{n^2} \sum_{i_1\neq j_1,i_2\neq j_2}( n \Sigma_{i_1 j_1} \Sigma_{i_2 j_2} +n \Sigma_{i_1 i_2} \Sigma_{j_1 j_2} +n\Sigma_{i_1 j_2} \Sigma_{i_2 j_1} + n(n-1) \Sigma_{i_1 j_1} \Sigma_{i_2 j_2})\nonumber\\
&= (\sum_{i\neq j} \Sigma_{ij})^2 + \frac{2}{n} \sum_{i_1\neq j_1,i_2\neq j_2} \Sigma_{i_1 i_2} \Sigma_{j_1 j_2}.
\end{align}

Also note that 
$$\sum_{i\neq j} \Sigma_{ii} \Sigma_{jj} = (\sum_{i=1}^m \Sigma_{ii})^2 - \sum_{i=1}^m \Sigma_{ii}^2.$$
Thus we obtain the following formula of $\E \| K_\theta - \Sigma \|_{NF}^2$ by plugging \eqref{eq:term1}-\eqref{eq:term5} to \eqref{eq:explicit1} and dividing $m$ on both sides:
\begin{align}\label{eq:explicit}
\frac{1}{m}\E \| K_\theta - \Sigma \|_F^2 & = \Big[ \frac{(n+2)(\theta^2 + m-1)}{n(\theta+m-1)^2}  - \frac{\theta-m-1}{\theta+m-1} - \frac{2\theta+m-2}{ (\theta+m-1)^2}  - \frac{\theta^2 (\theta-1)^2}{n(\theta+m-1)^2 (\theta+m-2)^2} \Big] \frac{1}{m}\sum_{i=1}^m\Sigma_{ii}^2 \nonumber\\
&\quad+\Big[ \frac{(2\theta+m-2)}{ (\theta+m-1)^2}  + \frac{\theta^2 (\theta-1)^2}{n(\theta+m-1)^2 (\theta+m-2)^2}  -\frac{2}{\theta+m-1}\Big] \frac{1 }{m}(\sum_{i=1}^m \Sigma_{ii})^2 \nonumber\\
&\quad+\Big[ \frac{2(2\theta+m-2)}{ n(\theta+m-1)^2} + \frac{(n+1)\theta^2 (\theta-1)^2}{n (\theta+m-1)^2 (\theta+m-2)^2} +1-\frac{2\theta(\theta-1)}{(\theta+m-1) (\theta+m-2)}\Big] \frac{1}{m}\sum_{i\neq j} \Sigma_{ij}^2 \nonumber\\
&\quad+\Big[\frac{ 2 (n+1)(\theta-1)^2 (2\theta+m-2) }{n(\theta+m-1)^2 (\theta+m-2)^2}- \frac{4(\theta-1)}{ (\theta+m-1) (\theta+m-2)} \Big]\frac{1}{m} \sum_{i=1}^m (\sum_{j\neq i} \Sigma_{ij})^2 \nonumber\\
&\quad+\Big[ \frac{ 2(\theta-1) (2\theta+2m-3) + m(m-1)}{(\theta+m-1)^2 (\theta+m-2)^2} - \frac{2}{(\theta+m-1) (\theta+m-2)}\Big] \frac{1}{m}(\sum_{i\neq j} \Sigma_{ij})^2 \nonumber\\
&\quad+ \frac{ 2 (\theta-1)^2 (2\theta+m-2)}{n(\theta+m-1)^2 (\theta+m-2)^2} \big( \frac{1}{m}\sum_{i=1}^m\sum_{j_1,j_2\neq i} \Sigma_{ii} \Sigma_{j_1 j_2} \big) \nonumber\\
&\quad+ \frac{ 2(\theta-1) (2\theta+2m-3) + m(m-1)}{n(\theta+m-1)^2 (\theta+m-2)^2} \frac{2}{m} \sum_{i_1\neq j_1,i_2\neq j_2} \Sigma_{i_1 i_2} \Sigma_{j_1 j_2}.
\end{align}

\end{document}